\documentclass[a4paper]{article}
\usepackage{amsthm}
\usepackage{amssymb}
\usepackage{amsmath}
\usepackage{verbatim,enumerate,url}
\usepackage[numbers]{natbib}
\usepackage{graphicx,multicol}

\renewcommand{\thefootnote}{\fnsymbol{footnote}}

\newtheorem{theorem}{Theorem}[section]
\newtheorem{lemma}[theorem]{Lemma}

\newtheorem{proposition}[theorem]{Proposition}
\newtheorem{remark}[theorem]{Remark}

\newtheorem{example}[theorem]{Example}

\newtheorem*{example*}{Example}

\newtheorem*{remark*}{Remark}

\newtheorem{corollary}[theorem]{Corollary}
\newtheorem*{corollary*}{Corollary}

\newtheorem{definition}[theorem]{Definition}

\newtheorem*{definition*}{Definition}

\newtheorem*{notation*}{Notation}

\numberwithin{equation}{section}

\gdef\myletter{}

\let\savetheequation\theequation
\def\theequation{\savetheequation\myletter}

\def\supp{\mathrm{supp}}

\newcommand{\CC}{{\mathbb C}}
\newcommand{\QQ}{{\mathbb Q}}

\newcommand{\RR}{{\mathbb R}}
\newcommand{\ZZ}{{\mathbb Z}}

\newcommand{\PP}{{\mathbb P}}

\newcommand{\NN}{{\mathbb N}}

\newcommand{\calN}{\mathcal{N}}
\newcommand{\lt}{\textsc{lt}}

\def \bar{\overline}
\def \hat{\widehat}

\def \b0{{\bf 0}}
\def\bV{{\bf V}}
\def\bI{{\bf I}}

\def\calL{\mathcal{L}}

\long\def\symbolfootnote[#1]#2{\begingroup%
\def\thefootnote{\fnsymbol{footnote}}\footnote[#1]{#2}\endgroup}

\newcommand{\ldeg}{\deg_{\calL}}

\begin{document} 

\title{Lelong classes of plurisubharmonic functions on an affine variety}
\author{Jesse Hart and Sione Ma`u}

\maketitle


\begin{abstract}
We study the Lelong classes $\calL(V),\calL^+(V)$ of psh functions on an affine variety $V$.  We compute the Monge-Amp\`ere mass of these functions, which we use to define the degree of a polynomial on $V$ in terms of  pluripotential theory (the \emph{Lelong degree}).  We compute the Lelong degree explicitly in a specific example.  Finally, we derive an affine version of B\'ezout's theorem.
\end{abstract}

\section{Introduction}

Although they are algebraic objects, complex polynomials can also be studied as entire holomorphic functions that satisfy certain growth restrictions.  The tools of complex analysis (such as the Cauchy integral formula) can be used to study their deeper properties.  An example of this in one variable is the standard complex analytic proof of the fundamental theorem of algebra.  In turn, the analysis of holomorphic functions involves the use of plurisubharmonic (psh) functions.  In one dimension, these are the classical subharmonic functions of potential theory in the plane; in higher dimensions psh functions satisfy pluripotential theory,  a nonlinear generalization based on the complex Monge-Amp\`ere operator.  The prototypical example of a psh function is $\log|F|$, where $F$ is a holomorphic mapping.     

Back in one complex variable, $z=x+iy$, a polynomial
$$
p(z) = a_dz^d + a_{d-1}z^{d-1} + \cdots + a_1z + a_0
$$
is classified by its degree $\deg(p):=d$.  This gives the order of growth of $p$ at infinity, as well as its  number of zeros, counting multiplicity, by the fundamental theorem of algebra.  The degree can also be given in terms of the subharmonic function $u:=\log|p|$ in a couple of ways:
\begin{eqnarray}
\deg(p) &=& \sup \Bigl\{c>0\colon \lim_{|z|\to\infty} \tfrac{1}{c}u(z)-\log|z|=O(1), \Bigr\} \label{eqn:i1} \\
2\pi\deg(p) &=& \int dd^cu.  \label{eqn:i2}
\end{eqnarray} 

  For a function $\varphi$ of class $C^2$, $dd^c\varphi = (\Delta \varphi) dx\wedge dy$ where $\Delta=\partial^2/\partial x^2 + \partial^2/\partial y^2$ is the Laplacian, and the operator $dd^c$ extends to subharmonic functions as a positive measure.  Equation (\ref{eqn:i1}) simply reformulates the notion of growth, while (\ref{eqn:i2}) follows from the fundamental theorem of algebra, writing $p(z)=c\prod_i (z-a_i)$ and using the fact that $dd^c\log|z-a_i|=2\pi\delta_{a_i}$, where $\delta_a$ denotes the discrete probability measure on $\{a\}$. 

The situation is more complicated in several variables.  Instead of a single polynomial, one considers a system of polynomial equations $p_1(z)=\cdots=p_k(z)=0$, $z\in\CC^N$.  An underdetermined system has infinitely many solutions that form an algebraic set  $V=\bV(p_1,\ldots,p_k)$.  Consider adding another polynomial $p_{k+1}$; then finding a solution to  $p_1=\cdots=p_k=p_{k+1}=0$ is equivalent to finding a solution to 
$$p_{k+1}(z) = 0, \quad z\in V.$$ 
Polynomials on $V$ can be studied algebraically as elements of a ring $\CC[V]$, or as holomorphic functions on $V$, or more generally, in terms of pluripotential theoretic objects: psh functions and closed positive currents.  
Pluripotential theory on holomorphic varieties in $\CC^N$ was studied by Sadullaev \cite{sadullaev:estimate} and Zeriahi (\cite{zeriahi:criterion}, \cite{zeriahi:pluripotential}).  
 We should also mention that a general pluripotential theory on complex manifolds has been developed (by Demailly, Guedj/Zeriahi and others) that has numerous applications in complex and algebraic geometry; see   \cite{demailly:applications} for a survey.

In this paper, we study the Lelong classes $\calL(V),\calL^+(V)$ of psh functions on an affine variety\footnote{i.e., an irreducible algebraic subset of $\CC^N$.} $V\subset\CC^N$.  These are functions of (at most) logarithmic growth as $|z|\to\infty$ on $V$.  We obtain the formula
\begin{equation}\label{eqn:mass1}
\int_V (dd^cu)^{m} = d(2\pi)^m  \quad \hbox{for all } u\in\calL^+(V),
\end{equation}
where  $m=\dim(V)$, $d=\deg(V)$, and $(dd^c)^m$ is the $m$-th exterior power of $dd^c$, or \emph{complex Monge-Amp\`ere operator}.  Here, we use the complex analytic definition of degree and dimension (as in e.g. \cite{gunning:introduction} or \cite{shabat:introduction}) in terms of a  branched covering projection.  

In Section \ref{sec:noether} we construct the projection explicitly for an affine variety; this is a standard construction in commutative algebra, and provides good coordinates for computation (a \emph{Noether presentation}).   
In Section \ref{sec:lelong} we introduce the Lelong classes $\calL(V),\calL^+(V)$ of psh functions, and in Section \ref{sec:comparison}, we derive  formula (\ref{eqn:mass1}).  To carry out our computations, we adapt  some standard convergence and comparison theorems for the complex Monge-Amp\`ere operator in $\CC^N$.

In Section \ref{sec:degree} we introduce the Lelong degree of a polynomial $p\in\CC[V]$ as a generalization of (\ref{eqn:i2}).  It can also be interpreted as a Lelong number for the current $dd^cp$.  We use the quantity  $d(2\pi)^m$ derived in the previous section to normalize the degree. 
 In Section \ref{sec:curves} we compute the Lelong degree explicitly on an algebraic curve $V=\bV(P)\subset\CC^2$.  Although we only carry out the computation for the specific polynomial $q(x,y)=y$ in this paper, we hope to expand this to a method for computing $\deg_{\calL(V)}(q)$ for any $q\in\CC[V]$, using a Newton polygon associated to $V$.  Our computation shows that $\deg_{\calL(V)}(q)$ is a rational number that gives the average growth of $q$ along the branches of $V$ as $|z|\to\infty$.  The next step would be to generalize such a method to higher-dimensional varieties.

In the last section, we prove an affine version of B\`ezout's theorem.  We relate it to the classical B\`ezout theorem in projective space via an example in $\CC^2$.

\subsection{Notation}

We recall some standard notation in computational algebraic geometry and several complex variables. 

 Write $\langle f_1,\ldots,f_n\rangle$ for the ideal generated by elements $f_j$ of a polynomial ring $\CC[z]=\CC[z_1,\ldots,z_N]$, and for $S\subset\CC[z]$, write $\langle S\rangle$ for the ideal generated by $S$.  Also, define the algebraic sets 
\begin{eqnarray*}
\bV(f_1,\ldots,f_j) &:=& \{z\in\CC^N: f_1(z)=\cdots=f_j(z)=0\}, \\
 \bV(S)&:=& \{z\in\CC^N: f(z)=0 \hbox{ for all }f\in S\}.
\end{eqnarray*}
For $V\subset\CC^N$, define the ideal
$$
\bI(V):= \{f\in\CC[z]: f(z)=0 \hbox{ for all } z\in\ V\}.
$$

For $z=(z_1,\ldots,z_N)$, $z_j=x_j+iy_j$, we have 
$$d=\partial +\bar\partial, \ d^c=i(\bar\partial-\partial), \ dd^c = 2i\partial\bar\partial,$$
where $\partial(\cdot) = \sum_{j=1}^N\partial/\partial z_j(\cdot)\wedge dz_j$ and $\bar\partial(\cdot) = \sum_{j=1}^N\partial/\bar\partial z_j(\cdot)\wedge d\bar z_j$.  

\smallskip

For a psh function $u$, $dd^cu$ is a \emph{positive $(1,1)$-current}, i.e., a linear functional on smooth, compactly supported $(N-1,N-1)$-forms such that $\langle dd^cu,\varphi\rangle\geq 0$ if $\varphi$ is a \emph{strongly positive form}.\footnote{An example of such a form is $\varphi=f\beta^{n-1}$, where $f$ is a non-negative test function and  $\beta=\sum_j dx_j\wedge dy_j$.  See e.g. \cite{klimek:pluripotential}, chapter 3.}  Here we write $\langle \cdot,\cdot\rangle$ to denote the pairing of a current and a test form.

\section{Noether Presentation}\label{sec:noether}

In this section we construct good coordinates for computation on an affine variety in $\CC^N$.   First,  we recall the {grevlex\footnote{graded reverse lexicographic} monomial ordering} on polynomials in $\CC[z]=\CC[z_1,\ldots,z_N]$.

\begin{definition}\label{def:grevlex} \rm
The \emph{grevlex monomial ordering} is the ordering $\prec$  in which $z^{\alpha}\prec z^{\beta}$ if
\begin{enumerate}
\item either $|\alpha|<|\beta|$; or
\item $|\alpha|=|\beta|$ and there exists $j\in\{1,\ldots,N\}$  such that $\alpha_j<\beta_j$ and $\alpha_k=\beta_k$ for all $k<j$.
\end{enumerate}
\end{definition}
Here, we are using standard multi-index notation, $\alpha=(\alpha_1,\ldots,\alpha_N)$ and $|\alpha|=\alpha_1+\cdots+\alpha_N$; and similarly for $\beta$.  

\medskip
  
Denote by $\lt(p)$ the leading term of a polynomial $p$ with respect to the grevlex ordering.  Recall that a \emph{Gr\"obner basis} $\{f_1,\ldots,f_N\}$ of an ideal $I\subset\CC[z]$ is a collection of polynomials  satisfying
$$
I=\langle f_1,\ldots,f_N\rangle \ \hbox{ and }\   \langle\lt(I)\rangle = \langle \lt(f_1),\ldots,\lt(f_N)\rangle;
$$
here $\langle \lt(I)\rangle$ is the ideal generated by the monomials $\lt(I) = \{\lt(p): p\in I\}$.   

\begin{definition}\label{def:normalform} \rm 
The \emph{(grevlex) normal form} of a polynomial $p$ (with respect to $I$)  is the unique polynomial   $r(z) = \sum_{\alpha}r_{\alpha}z^{\alpha}$ for which the following properties hold:
\begin{enumerate}
\item $p = \sum_{j=1}^s q_jf_j + r$ where $\{f_1,\ldots, f_s\}$ is a Gr\"obner basis of $I$; and 
\item   $z^{\alpha}\not\in\langle\lt(I)\rangle$ whenever $r_{\alpha}z^{\alpha}$ is a nonzero term of $r$. 
\end{enumerate}
\end{definition}
The existence and uniqueness of $r$ follows from the generalized division algorithm for multivariable polynomials and the fact that the divisors are a Gr\"obner basis.  It provides a standard polynomial representative of elements of $\CC[z]/I$, which may be used to give a well-defined notion of degree.

\begin{definition}\rm
Let $V\subset\CC^N$ be an affine variety and $p\in\CC[z]$.  Then we define the \emph{degree of $p$ on $V$} by $\deg_V(p):=\deg(r)$, where $r$ is the normal form of $p$ with respect to $\bI(V)$.
\end{definition}
Note that the condition $|\alpha|=|\beta|$ in the definition of grevlex easily implies  
$\deg_V(p)\leq \deg(p).$

\smallskip

  Given an invertible complex linear map $L\colon\CC^N\to\CC^N$,  let
$L^*\colon\CC[z]\to\CC[z]$ be defined by $p(z)\mapsto p(L(z))$.  We seek a change of coordinates such that the grevlex normal form becomes particularly simple.  The following is a standard result in commutative algebra (Noether normalization).

\begin{theorem}\label{thm:noether}
Let $V\subseteq\CC^N$ be an (irreducible) affine variety and $I=\bI(V)$.  Then there is a non-negative integer $m\leq  N$ such that for a generic complex linear change of coordinates $z\mapsto L(z)$ we have the following.
\begin{enumerate}
\item Write $z=(x_1,\ldots,x_m,y_1,\ldots,y_{N-m})=:(x,y)$ and $J:=L^*(I)$.  Then the canonical map $\CC[x]\to \CC[x,y]/J$ induced by the inclusion $\CC[x]\subseteq\CC[x,y]$ is injective and finite (i.e., exhibits $\CC[x,y]/J$ as a finite extension of $\CC[x]$.)
\item For each $j=1,\ldots,N-m$ there is a $d_j\in\NN$ and an irreducible polynomial $g_i\in J$ of total degree $d_j$ of the form
\begin{equation}\label{eqn:gj}
g_j(x,y) \  = \  y_j^{d_j} -   \sum_{k=0}^{d_j-1} g_{jk}(x,y_1,\ldots,y_{j-1})y_j^k 
\end{equation}
 i.e.,  $\deg(g_{jk}) + k\leq d_j$ for each $k=0,\ldots,d_j-1$.  
\end{enumerate}
\end{theorem}

\begin{proof}
The proof is by induction on $N$; we follow the argument in \cite{greuelpfister:singular}.  The case $N=1$ is almost trivial: $I=\langle p\rangle$ for some monic polynomial $p(z)=z^d+(\hbox{lower terms})$, and $\CC[z]/I$ is a finite extension of $\CC$.\footnote{In this case, $z$ is the `$y$' variable and there is no `$x$' variable.}   

For general $N$, let $f\in\ I$ be a polynomial of degree $\deg(f):=d\geq 1$, with $\hat f$ its leading homogeneous part.  Then for $L(z)=(c_1\cdot z,\ldots,c_N\cdot z)$ an invertible linear map, where $c_j\cdot z:=c_{j1}z_1+\cdots + c_{jN}z_N$, we have 
\begin{equation*}
\hat f(L(z)) \  = \  \hat f(c_N)z_N^d  \ + (\hbox{lower terms in } z_N) .
\end{equation*}
As long as the generic condition $\hat f(c_N)\neq 0$ holds, we may define $$\tilde g(z):= \hat f(c_N)^{-1}f(L(z)),$$ and $\tilde g$ has the property (\ref{eqn:gj}).  The natural map $\CC[z']\to\CC[z',z_N]/\langle \tilde g \rangle$ is injective and finite, because $\tilde g$ is a monic polynomial in $z_N$ with coefficients in $\CC[z']$.  It is also easy to see that elements of $J=L^*(I)$ map to elements of $J$.  Hence after quotienting out by $J$, the map $\CC[z']/J_0\to\CC[z]/J$ is also injective and finite (where $J_0=\CC[z']\cap J$).  If $J_0=\langle 0\rangle$ then we are done with $m=N-1$ and $g_1=\tilde g$.

Otherwise, applying induction to the affine variety $V_0=\bV(J_0)\subseteq\CC^{N-1}$, we obtain an  integer $m\leq N$ such that for a generic linear map $L_0:\CC^{N-1}\to\CC^{N-1}$, the conclusion of the theorem holds, say, with polynomials $g_{1},\ldots, g_{N-1-m}$.  

The theorem then holds for $N$ with the linear change of coordinates 
$$(z',z_N)\mapsto L(L_0(z'),z_N)$$ and the polynomials
$
g_{1},\ldots,g_{N-1-m}, g_{N-m}$, where $g_{N-m}(z):=\tilde g(L_0(z'),z_N).$

Since $V$ is irreducible, $I$ is prime.  If a $g_j$ constructed above is reducible, it must contain an irreducible fact in $I$.  We replace $g_j$ by this irreducible factor.
\end{proof}  


\begin{definition}\rm
Given an affine variety $V$, coordinates $(x,y)$ that satisfy Theorem \ref{thm:noether} will be called a \emph{Noether presentation for $V$}.
\end{definition}

A Noether presentation has a geometric interpretation.  First, recall the following definition (cf., \cite{gunning:introduction}).
\begin{definition}\rm
Let $X_1$ and $X_2$ be analytic varieties.  Then a surjective holomorphic map $\varphi: X_1\to X_2$ \emph{exhibits $X_1$ as a branched covering of $X_2$} if there is a dense open subset $\Omega$ of the regular part of $X_2$ such that for each $x\in\Omega$ there is some neighborhood $U$ of $x$ for which $\varphi^{-1}(U)$ is a union of disjoint open subsets of the regular part of $X_1$. The branched covering is \emph{locally $d$-sheeted (resp. finite)} if the number of these sets is $d$ (resp. finite).  We set  $\deg_x(\varphi):=d$ to be the \emph{degree of the covering map $\varphi$ at $x$}.  The set $X_2\setminus\Omega$ is called the \emph{branch locus}. \end{definition}

 Now $x\mapsto\deg_x(\varphi)$ is a locally constant function in $x$ (e.g. by elementary complex analysis).  Hence when $X_2$ is connected, it is a global constant so we may write $\deg(\varphi)$ (independent of $x$).

\begin{proposition}
Let $(x,y)$ be a Noether presentation for an irreducible affine variety $V\subset\CC^N$. 
\begin{enumerate}
\item  There exists a constant $A>0$ such that 
\begin{equation} \label{eqn:yx}
\|y\| \leq  A(1+\|x\|), \ \hbox{ for all } (x,y)\in V.
\end{equation}
\item If $P:\CC^m\times\CC^{N-m}\ni(x,y)\mapsto x\in\CC^m$ is the projection, then $P(V)=\CC^m$ and the restriction $P\bigl|_V$ exhibits $V$ as a finite branched covering over $\CC^m$.  (Hence $V$ is a complex manifold of dimension $m$ away from the branch locus.)
\end{enumerate}
\end{proposition}

\begin{proof}
We use induction on $N-m$, i.e., the number of polynomials $g_1,\ldots,g_{N-m}$ given by Theorem \ref{thm:noether} that define $V$.  (Note that the argument does \emph{not} depend on $N$, the dimension of the ambient space.) 

\bigskip 
\noindent {\bf Step 1: Base case ($N-m=1$).} 
\begin{enumerate} 
\item We have $(x,y)=(x_1,\ldots,x_{N-1},y)$, and $V=\bV(g_1)$ by hypothesis, where as in (\ref{eqn:gj}), 
\begin{equation}\label{eqn:yd}
g_1(x,y) \ = \ y^{d_1} - \sum_{k=0}^{d_1-1}g_{1k}(x)y^k =0  \ \hbox{ for all } (x,y)\in V.
\end{equation}
We claim that the above equation implies the inequality
\begin{equation}\label{eqn:|y|}
|y| \leq 2\max_{0\leq j\leq d_1-1} |g_{1j}(x)|^{1/(d_1-j)}  \ \hbox{ for all } (x,y)\in V.
\end{equation}
For if not,                                                                                                                                          
  \begin{eqnarray*}
|g_1(x,y)y^{-d_1}| &=& | 1 + g_{1(d_1-1)}(x)y^{-1} + \cdots + g_{10}(x)y^{-d_1}|   \\
&\geq & 1- \left( |g_{1(d_1-1)}(x)y^{-1}| + \cdots + |g_{10}(x)y^{-d_1}|     \right) \\
&\geq & 1- (2^{-1}+\cdots+ 2^{-d_1})>0
\end{eqnarray*}
which is a contradiction. Hence (\ref{eqn:|y|}) holds.                                                 

Fixing $j$, there exists a constant $C_j$ such that $|g_{1j}(x)|\leq C_j(1+\|x\|)^{d_1-j}$ for all $x\in\CC^n$, since $g_{1j}(x)$ is a polynomial of degree at most $d_1-j$.  Putting this into the right-hand side of (\ref{eqn:|y|}) and letting $A=\max_j C_j^{1/(d_1-j)}$, we obtain (\ref{eqn:yx}).

\item To show that $P$ maps $V$ onto $\CC^{N-1}$, fix $x\in\CC^{N-1}$. Then  (\ref{eqn:yd}) is a nonzero polynomial equation of degree $d_1$ in $y$; hence  by the fundamental theorem of algebra, has $d_1$ solutions (counting multiplicity); write 
\begin{equation}\label{eqn:g1x} g_1(x,y)=\prod_{j=1}^{d_1} (y-\alpha_j(x)). \end{equation}  This gives at least one point $(x,y)\in V$ that maps to $x$ under $P$; so $P(V)=\CC^{N-1}$.  The holomorphic implicit function theorem says that locally, $P$ has a local holomorphic inverse on $V$  at all points away from the subvariety $V\cap\{\frac{\partial}{\partial y}g_1=0\}$.  In fact, the local inverses are given by $\alpha_j(x)$ in (\ref{eqn:g1x}), and at such points $x$ the values of $\alpha_j(x)$ are distinct for each $j$.   Hence $V$ is a holomorphic branched cover over $\CC^{N-1}$, with $d_1$ sheets.
\end{enumerate}


\noindent {\bf Step 2: Induction.}   
Write $P$ as the composition $P_{2}\circ P_{1}$ given by 
$$
(x,y)=(x,y',y_{N-m})\in\CC^N\stackrel{P_{1}}{\longmapsto} (x,y')\in\CC^{N-1}   \stackrel{P_{2}}{\longmapsto} x\in\CC^{N-m}.
$$
Then $P_1(V)=:V_1$ is contained in the variety $\bV(g_1,\ldots,g_{N-m-1})\subset\CC^{N-1}$.  Here we are using the fact that the polynomials $g_1,\ldots,g_{N-m-1}$ of (\ref{eqn:gj})  are \emph{independent of the last coordinate $y_{N-m}$}. 

 On the other hand, if $(x,y')\in \bV(g_1,\ldots,g_{N-m-1})$ then the fundamental theorem of algebra applied to $s\mapsto g_{N-m}(x,y',s)$ gives the existence of $(x,y',y_{N-m})\in V$, so that we have the reverse containment $\bV(g_1,\ldots,g_{N-m-1})\subseteq V_1$.   Hence $V_1=\bV(g_1,\ldots,g_{N-m-1})$, which shows that $V_1$ is an affine variety in $\CC^{N-1}$ and $(x,y')$ is a Noether presentation for $V_1$.  To prove each part of the theorem we apply induction to $V_1$ and  $W:=\bV(g_{N-m})\subset\CC^N$.
\begin{enumerate}
\item Let $(x,y',y_{N-m})\in V\subset\bV(g_{N-m})$. Then $(x,y')\in V_1$ and there exist constants $A_1,A_2>0$ such that 
$$
\|y_{N-m}\|\leq A_1(1+\|(x,y')\|) \hbox{ and }   \|y'\| \leq A_2(1+\|x\|).
$$
Hence 
\begin{eqnarray*}
\|y\| \leq |y_{N-m}|+\|y'\| &\leq&   A_1(1+\|(x,y')\|) + A_2(1+\|x\|) \\
&\leq& A_1(1+\|x\|+\|y'\|) + A_2(1+\|x\|) \\
&\leq& (A_1(1+A_2)+A_2)(1+\|x\|) =:A(1+\|x\|).
\end{eqnarray*}

\item We have $P(V)=P_2(V_1)=\CC^m$. By induction applied to $V_1$,  the projection $P_2$ is locally biholomorphic away from a proper subvariety $V_s\subset V_1$.  Also, by the implicit function theorem, $P_1\colon W\to\CC^{N-1}$ is a local biholomorphism at each point of $W\setminus W_s$, where $W_s= \bV(\frac{\partial}{\partial y_{N-m}}g_{N-m})$.  Hence $P_1\colon V\to V_1$ is a local  biholomorphism away from $P_1^{-1}(V_s)\cup W_s$.   So $P=P_2\circ P_1$ is a local biholomorphism at each point of $V\setminus(P_1^{-1}(V_s)\cup W_s)$ which gives $V$ as a branched covering over $\CC^m$.  Clearly, the covering is finite, of degree $\deg(P_1)\cdot\deg(P_2)$. 
\end{enumerate}\end{proof}

\begin{remark}\rm
Clearly, $V_1$ is irreducible if $V$ is, and this enters into the proof of the second part.     
Irreducibility of $V$ is not necessary for Theorem \ref{thm:noether}; but if $V$ were reducible, some of the $g_j$s will be reducible polynomials.  The conclusion of the second part of the proposition may fail because:
\begin{itemize}
\item A component of $V$ covers $\CC^{m+k}$ rather than $\CC^m$.  When this occurs, there are nontrivial algebraic relations among certain factors of the $g_j$s. Hence there is a component of $V$ for which one can cut down (to $N-m-k$, say) the number of defining polynomials.
\item A component of $V$ is completely contained in $W_s$.  Then there are nontrivial algebraic relations involving factors of the $g_j$s and $\frac{\partial}{\partial y_{N-m}} g_{N-m}$.\end{itemize}

Irreducibility of $V$, and hence of each polynomial in  $G=\{g_1,\ldots,g_{N-m}\}$, also guarantees that $G$ is a Gr\"obner basis,\footnote{One can check, using (\ref{eqn:gj}) together with irreducibility, that $G$ satisfies the so-called \emph{Buchberger criterion} for a Gr\"obner basis.  See e.g. \cite{coxlittleoshea:ideals}, chapter 2.}  so that  $\langle\lt(I)\rangle=\langle y_1^{d_1},\ldots,y_{N-m}^{d_{N-m}}\rangle$.  
Hence a normal form has the structure
\begin{equation}\label{eqn:normalform}
p(x,y) = \sum_{\alpha\in \calN_y} p_{\alpha}(x)y^{\alpha}
\end{equation}
where $p_{\alpha}\in\CC[x]$ for each $\alpha$, and $\alpha$ is taken over the finite set $$\calN_y:=\{\alpha=(\alpha_1,\ldots,\alpha_{N-m})\in\ZZ^{N-m}: \alpha_j<d_j \hbox{ for each } j\}.$$
\end{remark}

Since $\CC^m$ is connected, the projection $P$ has a well-defined global degree, $\deg(P)$.  A similar connectedness argument can be used to show the following.  

\begin{lemma}
For any two Noether presentations with projections $P,\tilde P$, we have $\deg(P)=\deg(\tilde P)$.  
\end{lemma}

\begin{proof}
 A complex linear perturbation of $\CC^N$ (i.e. a linear map $L_{\epsilon}=I+\epsilon T$ where $\epsilon<<\|T\|$) takes a Noether presentation $(x,y)$ to a Noether presentation $(x_{\epsilon},y_{\epsilon})$, as long as $\epsilon$ is sufficiently small: we have  $(x_{\epsilon},y_{\epsilon})=(x,y) + \epsilon T(x,y)$. 
The local inverses $P^{-1},P_{\epsilon}^{-1}$  of the projections $P(x,y):=x$ and $P_{\epsilon}(x_{\epsilon},y_{\epsilon}):= x_{\epsilon}$ are holomorphic, and $\Phi_{\epsilon}:=P_{\epsilon}\circ L_{\epsilon}\circ P^{-1}$ is a local biholomorphic map that goes to the identity as $\epsilon\to 0$.    Given a local inverse $P^{-1}$, the composition   
$ L\circ P^{-1}\circ\Phi_{\epsilon}^{-1}$ is a nearby local inverse for $P_{\epsilon}$.  Vice versa, given a local inverse $P_{\epsilon}^{-1}$, a nearby local inverse of $P$ is given by $L^{-1}\circ P_{\epsilon}^{-1}\circ\Phi_{\epsilon}$.  Hence the number of local inverses is the same for $P$ and $P_{\epsilon}$, so $\deg(P)=\deg(P_{\epsilon})$.  

Now, fix some reference coordinate system and let $L$ be a linear map that transforms these coordinates to a Noether presentation with projection $P_L$.  If we identify the collection of all such linear maps $L$ (equivalently, $N\times N$ matrices) with $\CC^{N^2}$, the maps that give Noether presentations form a connected open subset of $\CC^{N^2}$.\footnote{Here we use the fact that the complement of a proper analytic subset of $\CC^{N^2}$ is connected.}  By the previous paragraph, $L\mapsto\deg(P_L)$ is a locally constant function on this set.  Hence $\deg(P_L)$ is a constant {independent of $L$}.  \end{proof} 

\begin{definition}\rm Define the \emph{degree of $V$} by $\deg(V):= \deg(P)$, where $P:V\to\CC^m$ is the projection in a Noether presentation of $V$. The \emph{dimension of $V$} is $\dim(V):=m$.
\end{definition}

We give some examples of Noether presentations.
\begin{example}\rm
Let $V\subset\CC^2$ be given by the equation $z_1z_2=1$.  Then $(x,y)=(z_1,z_2)$ is not a Noether presentation.  Similar to the proof of Theorem \ref{thm:noether}, write $$z_1=a_1x+b_1y, \quad z_2=a_2x+b_2y,$$ then the equation for $V$ transforms into $b_1b_2y^2+ (a_1b_2+a_2b_1)xy + a_1a_2x^2 =1$, and  $(x,y)$ is a Noether presentation if $b_1b_2\neq 0$.  
\end{example}   

\begin{example}\rm
Let $V\subset\CC^2$ be given by the equation $z_1^2=z_2^3$.  Then $(x,y)=(z_1,z_2)$  is a Noether presentation, while 
 $(x,y)=(z_2,z_1)$ is not.  Clearly, (\ref{eqn:yx}) fails for the latter because $|z_1|=|z_2|^{3/2}$.
\end{example}




\begin{remark}\rm
The existence of coordinates for which (\ref{eqn:yx}) holds at all points on an affine variety was already known by Sadullaev and Zeriahi.  Such an estimate also characterizes algebraicity. Suppose $W$ is an analytic subvariety of $\CC^N$, and suppose there exist constants $A,B>0$ such that  $$\|y\|\leq A(1+\|x\|)^B \ \hbox{ for all } (x,y)\in W.$$   Rudin has shown \cite{Rudin:geometric}  that $W$ must therefore be algebraic.
\end{remark}

\section{The Lelong Class} \label{sec:lelong}

Let $V$ be an affine variety in $\CC^N$.  Then a function $u\colon V\to[-\infty,\infty)$ is plurisubharmonic (psh) on $V$ if in a neighborhood of each point, $u$ is locally the restriction of a psh function in a local embedding into $\CC^N$.

\begin{remark}\label{rem:21}  \rm
The above notion is sufficient for this paper but a weaker notion is needed to get a class with good compactness properties (\cite{sadullaev:estimate}, \cite{zeriahi:pluricomplex}, \cite{dinhsibony:equidistribution}).  A function is \emph{weakly psh} on a complex space $X$ if it is upper semicontinuous on $X$ and psh in local coordinates about every regular point of $X$.  Both of these notions coincide when $X$ is {smooth}. 
\end{remark}  

The \emph{Lelong class $\calL(V)$} is the class of plurisubharmonic (psh) functions on $V$ of at most logarithmic growth:
$$
\calL(V) \  := \  \{ u \hbox{ psh on } V\colon \ \exists C\in\RR \hbox{ such that } u(z)\leq \log^+\|z\| + C, \forall z\in V\}. 
$$
We also define the class 
$$
\calL^+(V) \ := \  \{u\in \calL(V): \ \exists c\in\RR \hbox{ such that } u(z)\geq \log^+\|z\| + c, \forall z\in V\}.
$$
It is easy to see that these classes are invariant under complex linear changes of coordinates. 
Also, for a real constant $c>0$, write
 \[
c\calL(V):=\{ cu\colon u\in\calL(V)\},   \quad c\calL^+(V):= \{cu\colon u\in\calL^+(V)\}.   \]

\begin{proposition}\label{prop:L}
 Suppose that $z=(x,y)$ is a Noether presentation for an affine algebraic variety $V\subset\CC^N$.  Then 
\begin{enumerate}
\item $\log^+\|x\|\in \calL^+(V)$.
\item For any polynomial $p$ with $\deg_V(p)\geq 1$, we have $\displaystyle \frac{1}{\deg_V(p)}\log|p|\in\calL(V)$. 
\end{enumerate}
\end{proposition}

\begin{proof}
We need to show that the quantity $\log^+\|x\|-\log^+\|z\|$ is uniformly bounded from both sides.  The inequality $\log^+\|x\|\leq \log^+\|z\|$ is obvious, which gives an upper bound of zero.  For the lower bound, we use the property $\|y\|\leq A(1+\|x\|)$  for some $A>0$ to estimate
\begin{eqnarray*}
\log^+\|z\| = \log^+\|(x,y)\| &\leq&  \log^+(\|x\| + \|y\|)  \\
 &\leq&   \log^+(\|x\|+ A(1+\|x\|)) \\ &\leq&   \log^+( \tfrac{A}{A+1} +\|x\|) +\log^+(A+1)  \leq \log^+\|x\|+C,
\end{eqnarray*}
where we choose $C>0$ sufficiently large (depending on $A$) so that the last inequality holds for all $x\in\CC^M$.  This gives a lower bound of $-C$.

To prove the second item, let $d:=\deg_V(p)$.  Reduce $p$ to its normal form of degree $d$, which we will also denote by $p$, and let $\hat p$ denote the leading homogeneous part.  We have  
$$
p(x,y) = \hat p(x,y) + r(x,y) =  \sum_{\substack{|\beta|\leq d\\ \beta\in\calN_y}} h_{\beta}(x)y^{\beta} + r(x,y)
$$
with $h_{\beta}(x)$ homogeneous of degree $d-|\beta|$ for each $\beta$ and $\deg r(x,y)< d$.  We calculate that 
\begin{eqnarray*}
\frac{|h_{\beta}(x)y^{\beta}|}{\|x\|^d} \leq \frac{|h_{\beta}(x)|}{\|x\|^{d-|\beta|}}\left(\frac{\|y\|}{\|x\|}\right)^{|\beta|} 
&\leq&  \left| h_{\beta}\bigl(\tfrac{x}{\|x\|}\bigr) \right|\cdot C^{|\beta|}\left|1+\tfrac{1}{\|x\|}\right|^{|\beta|}  \\
&\leq& \|h_{\beta}\|_{B }(2C)^{|\beta|} \quad(\hbox{for   } \|x\|>1)  \\
&=:& C_{\beta},
\end{eqnarray*}
where in the second inequality $\|\cdot\|_B$ denotes the sup norm on $B$, the closed unit ball in $\CC^m$.  

If $(\alpha,\beta)\in\NN^m\times\NN^{N-m}$ are multi-indices with $|\alpha|+|\beta|<d$, then
$$
|x^{\alpha}y^{\beta}|\leq \|x\|^{|\alpha|}\|y\|^{|\beta|} \leq \|x\|^{|\alpha|}(1+\|x\|)^{|\beta|}
$$
so that 
$$\frac{|x^{\alpha}y^{\beta}|}{\|x\|^{d}} \leq \|x\|^{d-|\alpha|} + \|x\|^{d-|\alpha|-|\beta|} \longrightarrow 0 \ \hbox{ as } \|x\|\to\infty.$$
Hence $|r(x,y)|/\|x\|^d\to 0$ as $\|x\|\to\infty$.  Putting the above calculations together,  
$$
\frac{|p(x,y)|}{\|x\|^d} \leq \tilde C \hbox{ for sufficiently large } \|x\|>1
$$
where $\tilde C=1+\sum_{\beta}C_{\beta}$.  Thus $\frac{1}{d}\log|p(x,y)| \leq \log\|x\| + \tilde C/d$ for sufficiently large $\|x\|$.  It follows easily that $\frac{1}{d}\log|p(x,y)| \in\calL(V)$.
\end{proof}

The integer $\deg_V(p)$ is not the smallest value of $d$ permitting an inequality of the form $\frac{1}{d}\log|p(z)|\leq \log^+\|z\|+A$ for some $A\in\RR$.  The optimal bound may be a rational number; we will see this later, when studying the notion of \emph{Lelong degree}.

\begin{remark}\rm
 In \cite{zeriahi:pluricomplex} a complex space $X$ of dimension $m$ is said to be \emph{parabolic} if it admits a continuous psh exhaustion function $g:X\to[-\infty,\infty)$ for which $(dd^cg)^m=0$ off some compact subset of $X$.  The Lelong class $\calL(X,g)$ is then defined by 
$$
\calL(X,g) = \{ u \hbox{ {weakly} psh on }  X: \exists C \hbox{ such that } u(z)\leq g^+(z) + C, \forall z\in V \},
$$
where $g^+(z)=\max\{g(z),0\}$, and the notion of weakly psh is as in Remark \ref{rem:21}.  We define $\calL^+(X,g)$ similarly with the added condition $u(z)\geq g^+(z)+c$.  

 Proposition \ref{prop:L} shows that  any affine variety $V$ is a parabolic space with parabolic potential $\log\|x\|$, where the coordinates $(x,y)$ are a Noether presentation for $V$.  The classes $\calL(V)$, $\calL^+(V)$ correspond to the classes $\calL(V,\log\|x\|)$,  $\calL^+(V,\log\|x\|)$ in the notation of \cite{zeriahi:pluricomplex}.
\end{remark}

\section{Comparison theorems and Monge-Amp\`ere mass} \label{sec:comparison}

We will establish some comparison theorems for the complex Monge-Amp\`{e}re operator on an affine variety. The locally bounded theory of Bedford and Taylor \cite{bedfordtaylor:dirichlet} in $\CC^N$  is sufficient for this section.  In particular, we use the fact that any locally bounded psh function in $\CC^N$ can be approximated by a decreasing sequence of smooth psh functions; and if $T$ is a positive closed $(N-k,N-k)$-current, $u_1,\ldots,u_k$ are psh functions, and for each $j$ we have a monotone convergent sequence of psh functions, $u_{j}^{(n)}\nearrow u_j$ (or $u_{j}^{(n)}\searrow u_j$), then
$$
dd^cu_1^{(n)}\wedge\cdots\wedge dd^cu_k^{(n)} \wedge T \to dd^cu_1\wedge\cdots\wedge dd^cu_k\wedge T \quad \hbox{weak-}^* \hbox{ as } j\to\infty.
$$
 

\begin{theorem}
Let $\Omega\subset\CC^N$ and let $u,v$ be locally bounded psh functions on $\Omega$.  Let $T$ be a closed positive current of bidegree $(N-j,N-j)$ on $\Omega$ for some positive integer $j\leq N$.   If the set $A:=\{u<v\}\cap \supp(T)$ is a relatively compact subset of $\Omega$ then 
$$
\int_{\{u<v\}} (dd^cv)^j\wedge T \leq \int_{ \{u<v\} } (dd^cu)^j\wedge T.
$$
\end{theorem}

\begin{proof}
We will verify the theorem for $u,v$ continuous, adapting an argument of Cegrell (see \cite{cegrell:capacities} or \cite{klimek:pluripotential}, Section 3.7).  The hypotheses can be weakened to $u,v$ locally bounded by a standard argument using decreasing sequences of continuous psh approximants to $u,v$ on an open neighborhood of the closure of $A$.  

  Let $\Omega_1:=\Omega\cap\{u<v\}$,  and let $T$ be a positive closed $(N-1,N-1)$-current.   For $\epsilon>0$, define $v_{\epsilon}:=\max\{v-\epsilon,u\}$.  By continuity, $u=v$ on $\partial\Omega_1\cap\supp(T)$, and by hypothesis, the closure of this set is in $\Omega$.  Thus $\{u<v-\epsilon\}\cap\supp(T)$ is a relatively compact subset of $\Omega_1$.  We claim that 
\begin{equation}\label{eqn:3a}
\int_{\Omega_1} dd^cv_{\epsilon}\wedge T = \int_{\Omega_1} dd^cu\wedge T.
\end{equation}
Given $\epsilon>0$, let $\varphi$ be a smooth, compactly supported non-negative function in $\Omega_1$ such that $\varphi\equiv 1$ in a neighborhood of the closure of the set $\{v_{\epsilon}>u\}\cap\supp(T)$.  Then $dd^c\varphi\equiv 0$ on this set, so that $(\supp(dd^c\varphi)\cap\supp(T))\subset \{v_{\epsilon}=u\}$, and
$$
\int_{\Omega_1} \varphi dd^cv_{\epsilon}\wedge T  = \int_{\Omega_1} v_{\epsilon} dd^c\varphi\wedge  T = \int_{\Omega_1} udd^c\varphi\wedge T.
$$
Consequently (\ref{eqn:3a}) holds since $\varphi$ was arbitrary.

Now take a smooth, compactly supported $\psi$  on $\Omega_1$, with $0\leq\psi\leq 1$.  Then $v_{\epsilon}\nearrow v$ on $\Omega_1$, so that
$$
\int_{\Omega_1}\psi dd^cv\wedge T =\lim_{\epsilon\to 0} \int_{\Omega_1}\psi dd^cv_{\epsilon}\wedge T \leq \lim_{\epsilon\to 0} \int_{\Omega_1} dd^cv_{\epsilon}\wedge T
= \int_{\Omega_1} dd^cu\wedge T,
$$
using (\ref{eqn:3a}).  Since $\psi$ was arbitrary, the theorem follows when $j=1$.

An easy induction gives the theorem for higher powers of $j$.
\end{proof}


Recall that if  $M\subset\CC^N$ is a manifold of dimension $k$, then $[M]$ denotes its current of integration, i.e., the $(2N-k)$-current that acts on a test $k$-form $\varphi$ by
$$
\langle [M],\varphi\rangle := \int_M \varphi.
$$
A theorem of Lelong says that the current of integration $[V]$ over an analytic variety $V\subset\CC^N$ of pure dimension $m$  is a closed positive current of bidegree $(N-m,N-m)$.  As a consequence, we have the following.

\begin{corollary} \label{cor:32}
Let $V\subset\CC^N$ be an affine variety of dimension $m$ and let $j\in\{1,\ldots, m\}$.  Suppose $u,v$ are locally bounded psh functions on $V$ and $S$ is a closed positive current of bidegree $(m-j,m-j)$.   If $A:=V\cap \{u<v\}\cap\supp(S)$ is a bounded subset of $V$, then 
$$
\int_{A} (dd^cv)^j\wedge S \leq \int_{A} (dd^cu)^j\wedge S.
$$
\end{corollary}

\begin{proof}
Note that by definition, $u,v$ are restrictions to $V$ of locally bounded psh functions on an open set $\Omega\subset\CC^N$ (cf. Remark \ref{rem:21}).  For these extended functions, $\{u<v\}\cap V$ is relatively compact in $\Omega$, so we may apply the previous theorem with $T=[V]\wedge S$.
\end{proof}

\begin{theorem}\label{thm:33}
Suppose that $V\subset\CC^N$ is an affine variety of dimension $m$, and $u,v$ are locally bounded psh functions on $V$.  Suppose $u,v$ are bounded from below outside a bounded subset of $V$ and $v(z)=u(z) + o(u(z))$ as $\|z\|\to\infty$.  Let $j\in\{1,\ldots,m\}$. Then for any closed positive $(m-j,m-j)$-current $S$ with unbounded support,
$$
\int_V (dd^cv)^j\wedge S   \leq \int_V (dd^cu)^j\wedge S.
$$
\end{theorem}

\begin{proof} 
By adding a positive constant, we may assume without loss of generality that $u,v>0$ outside a bounded subset of $V$.  
   Let $\epsilon,c>0$.  Then $(1+ \epsilon) u(z) -c >u(z)+o(u(z))$  as $\|z\|\to\infty$, which implies that $(1+\epsilon) u(z)-c>v(z)$ as $\|z\|\to\infty$.  Thus the set $A_{\epsilon,c}:=\{(1+\epsilon)u-c<v\}\cap V$ is bounded and we may use Corollary \ref{cor:32} to obtain the inequality
$$
\int_{A_{\epsilon,c}} (dd^cv)^j\wedge S   \leq  (1+\epsilon)^j \int_V (dd^cu)^j\wedge S.
$$
Letting $c\to+\infty$, $$\int_V  (dd^cv)^j\wedge S   \leq  (1+\epsilon)^j \int_V (dd^cu)^j\wedge S.$$   
Finally, let $\epsilon\to 0$.
\end{proof}

We will also need the following corollary.
\begin{corollary}\label{cor:L2}
Suppose $u-v$ is $O(1)$ and $S$ is a positive closed $(m-1,m-1)$-current with the property that $\int_V dd^cw\wedge S$ is finite for some $w\in\calL^+(V)$.  Then 
$$
\int_V dd^cu\wedge S = \int_V dd^cv\wedge S.
$$
\end{corollary}

\begin{proof}
Take $w\in\calL^+(V)$ as hypothesised, and let $\epsilon,c>0$.  Define $u_{\epsilon,c}:=u+\epsilon w-c$.  Then the set $A_{\epsilon,c}:=\{u_{\epsilon,c}<v\}$ is bounded, so that
$$
\int_{u_{\epsilon,c}} dd^cv\wedge S \leq \int_V dd^cu_{\epsilon,c}\wedge S = \int_V dd^cu\wedge S + \epsilon\int_V dd^cw\wedge S.
$$
As before, let $c\to+\infty$ and  $\epsilon\to 0$ to obtain $\int_V dd^cv\wedge S \leq \int_Vdd^cu\wedge S$.  

The reverse inequality is obtained by the same argument with the roles of $u$ and $v$ swapped.
\end{proof}

\begin{corollary}\label{cor:L}
Let $u,v$ be psh on $V$, bounded from below, $u-v=O(1)$, and both functions go to infinity as $|z|\to\infty$.  Let $S$ be a positive closed $(m-j,m-j)$-current with unbounded support.   Then 
\begin{equation}\label{eqn:S}
\int_V (dd^cu)^j\wedge S = \int_V (dd^cv)^j\wedge S.
\end{equation}
In particular,  
 \begin{equation}\label{eqn:uv} \displaystyle \int_V (dd^cu)^m = \int_V (dd^cv)^m. \end{equation}
If the value of (\ref{eqn:uv}) is a  nonzero constant, it is equal to  
\begin{equation}\label{eqn:uv2} \displaystyle\int (dd^cu)^j\wedge(dd^cv)^{m-j} \end{equation} for each 
$j\in\{0,\ldots,m\}$.  In particular, this is true for $u,v\in\calL^+(V)$.  
\end{corollary}

\begin{proof}
Since $u,v$ go to infinity as $|z|\to\infty$,  $u-v$ is both $o(u(z))$ and $o(v(z))$. We need only show `$\geq$' in (\ref{eqn:S}), (\ref{eqn:uv}), and (\ref{eqn:uv}); the opposite inequality follows by reversing the roles of $u$ and $v$.  

The inequality for (\ref{eqn:S}) follows from Theorem \ref{thm:33}, and that for (\ref{eqn:uv}) follows as a corollary upon taking $S\equiv 1$.  For $j\in\{1,\ldots,m-1\}$ we also obtain the inequality for (\ref{eqn:uv2}) by taking $S=(dd^cv)^{m-j}$, as long as the support of this current is unbounded; we verify this next.


Suppose on the contrary that the support is bounded, i.e.,  
\begin{equation}\label{eqn:support} (dd^cv)^{m-1}=0 \ \hbox{ on }  \ V\setminus\{\|z\|\leq R\}  \ \hbox{ for some } R>0.\end{equation}
  By hypothesis, we may choose such an $R$ sufficiently large for which  $\int_{V_R} (dd^cu)^{m} \neq  0$, where  $V_R = \{z\in V: \|z\|\leq R\}$.  Let $\varphi\geq 0$ be a smooth, compactly supported function that is identically 1 on a neighborhood of $V_R$.  Then $\supp(dd^c\varphi)\subset V\setminus V_R$, so integration by parts yields 
\begin{eqnarray*}
0 \ \neq \  \int_ {V_R} (dd^cv )^m \ \leq \  \int_{V}\varphi(dd^cv)^{m} &=& \int_V v dd^c\varphi\wedge (dd^cv)^{m-1} \\
 &=& \int_{V\setminus V_R}  v dd^c\varphi\wedge (dd^cv)^{m-1} \  =  \  0, 
\end{eqnarray*}
using (\ref{eqn:support}).  This is a contradiction.
\end{proof}

As we shall see, $\int_V(dd^cv)^m$ is a positive constant for $v\in\calL^+(V)$.  The common value of the integrals (\ref{eqn:uv}), (\ref{eqn:uv2}) will be called the \emph{(Monge-Amp\`ere) mass of $\calL^+(V)$}.  Its exact value is obtained by choosing a convenient function in the class that can be integrated explicitly.  

 In Chapter 5 of \cite{klimek:pluripotential}, Klimek computes  $\int (dd^c(\log(1 + \|z\|)))^N$ in polar coordinates to get the mass of $\calL^+(\CC^N)$.  We give an alternative computation  using a `max' formula.
\begin{proposition}
\label{lem:43}
Let $0\leq k\leq m$ and suppose the functions $u_1,\ldots,u_{k+1}$ are pluriharmonic (i.e. $dd^cu_j=0$ for all $j$) on a domain $\Omega\subseteq\CC^m$.  Let $u(z):=\max_j u_j(z)$.  Then 
\begin{equation}\label{eqn:MAformula}
(dd^cu)^k = [S]\wedge d^c(u_1-u_2)\wedge d^c(u_2-u_3)\wedge\cdots\wedge d^c(u_{k}-u_{k+1}), 
\end{equation}
where $S=\{z\in\CC^m: u_1(z)=u_2(z)=\cdots=u_{k+1}(z)\}$, as long as $S$ is a smooth $(2m-k)$-dimensional manifold.
 \qed
\end{proposition}
The formula is a direct corollary of the main theorem in \cite{bedfordmau:complex} 
(i.e. without the `$dd^c$' terms, which vanish in this case).  The pairing of $[S]\wedge d^c(u_1-u_2)\wedge\cdots\wedge d^c(u_k-u_{k+1})$ with a smooth $(N-k,N-k)$-current $\omega$ is given by  
$$\langle [S]\wedge d^c(u_1-u_2)\wedge\cdots\wedge d^c(u_k-u_k+1),\omega\rangle = \int_S d^c(u_1-u_2)\wedge\cdots\wedge d^c(u_k-u_{k-1})\wedge \omega,$$ which means $S$ must be oriented so that the integral is non-negative for $\omega\geq 0$.

\begin{example}\label{ex:36} \rm
For $z=x+iy\in\CC$, we have $\log^+|z|=\max\{\log|z|,0\}$, with $S=\{|z|=1\}$, so $dd^c\log^+|z|=[\{|z|=1\}]\wedge d^c\log|z|$.  We compute 
\begin{eqnarray*}
d^c\log|z| = d^c\left( \tfrac{1}{2}\log(x^2+y^2)\right) &=& \tfrac{1}{x^2+y^2}(xdy - ydx ) \\
 &=& \tfrac{1}{r^2}(r^2\cos^2\theta(d\theta)+r^2\sin^2\theta d\theta 
 \ = \   d\theta,
\end{eqnarray*}
where $z=re^{i\theta}$.   Hence $dd^c\log^+|z| = [\{|z|=1\}]\wedge d\theta$, angular measure on the unit circle.
\end{example}

More generally, consider the function in $\calL^+(\CC^N)$ given by  $$u(z)=\max\{\log|z_1|,\ldots,\log|z_N|,0\}.$$ Then $S$ is the torus $\{|z_1|=\cdots=|z_N|=1\},$  and we have
\begin{eqnarray} 
(dd^cu)^N 
&=& [S]\wedge d^c\log|z_1|\wedge\cdots\wedge d^c\log|z_N|  \nonumber  \\
&=& [S]\wedge d\theta_1\wedge\cdots\wedge d\theta_N,   \label{eqn:37}
\end{eqnarray}
where $z_j=r_je^{i\theta_j}$ for each $j$.  
Then $\int (dd^cu)^N = \int_S d\theta_1\wedge \cdots\wedge d\theta_N = (2\pi)^N$, showing that 
\begin{equation}\label{eqn:massL+}
\hbox{\it the mass of $\calL^+(\CC^N)$ is $(2\pi)^N$.}
\end{equation}

Using (\ref{eqn:massL+}) together with a Noether presentation, we can prove the following theorem.


\begin{theorem}\label{thm:L}
For an affine variety $V$ of dimension $m$,  the mass of $\calL^+(V)$ is $(\deg V)(2\pi)^m$.
\end{theorem}

\begin{proof}
By Corollary \ref{cor:L} it suffices to compute the Monge-Amp\`ere mass of any function in $\calL^+(V)$.  By Proposition \ref{prop:L}(1), we  may take $u\in\calL^+(V)$ to be the function $u(x,y)=\log^+\|x\|$, where $(x,y)$  is a Noether presentation; and being independent of $y$, it is naturally identified as a function in $\calL^+(\CC^m)$.  Let $B$ be the closed unit ball in $\CC^m$.  We will first prove the theorem under the condition that
\begin{itemize} 
\item[($\star$)] \it The projection $(x,y)\stackrel{P}{\longmapsto} x $ is a local biholomorphism from a neighborhood of $P^{-1}(B)\cap V$  to a neighborhood of $B$, and $P^{-1}(B)$ is a union of $d$ disjoint sets.
\end{itemize}
(In other words, $B$ avoids the branch locus of the projection.)  Then on $\CC^m\setminus B$, the function $\log^+\|x\|$ satisfies $(dd^c\log^+\|x\|)^m=0$; hence $u$ satisfies 
\begin{equation}\label{eqn:ddcu}
(dd^cu)^m=0 \hbox{ in a neighborhood of each  } (x,y)\in P^{-1}(\CC^m\setminus B)\setminus E 
\end{equation}
where $E$ is the algebraic set given by the union of the singular points of $V$ and the branch points of $P$.  
Thus the $(N,N)$-current $T:= [V]\wedge (dd^cu)^m$ is zero on the set $(\CC^N\setminus (P^{-1}(B))\setminus  E$. We claim that it must therefore be zero on $E$ as well.  For if not, then for any test function $\varphi$, 
\begin{equation}\label{eqn:E}
0< \int_{\CC^N\setminus P^{-1}(B)}\varphi  T \  = \    \int_E  \varphi T = \int_E \varphi (dd^cu)^m = \lim_{k\to\infty} \int_E \varphi (dd^cu_{k})^m,
\end{equation}
where $\{u_k\}$ is a decreasing sequence of smooth psh functions, $u_k\searrow u$.  Since each $u_k$ is smooth, for each $k$ we may interpret $\int_E \varphi (dd^cu_{k})^m$ classically as the integral of a smooth $(m,m)$-form over $E$, an analytic set of complex dimension $<m$, which is always zero.  Hence the limit on the right-hand side of (\ref{eqn:E}) is zero, a contradiction. 

 Thus $T$ is supported on $V\cap P^{-1}(\overline B)$, which means that 
\begin{eqnarray*} \int_{V}(dd^cu)^m =   \int_{V\cap P^{-1}(\bar B)}(dd^cu)^m &=& \sum_{j=1}^{\deg V} \int_{\bar B}(dd^c\log^+\|x\|)^m \\
&=& \deg V \int_{\bar B}(dd^c\log^+\|x\|)^m
\end{eqnarray*}
where condition $(\star)$ is used to get the second equality in the first line. Observe now that as a global function on $\CC^m$,  $\log^+\|x\|\in \calL^+(\CC^m)$ and is maximal on $\CC^m\setminus \bar B$, so that $$\int_{\bar B}(dd^c\log^+\|x\|)^m = \int_{\CC^m}(dd^c\log^+\|x\|)^m=(2\pi)^m,$$ by 
(\ref{eqn:massL+}).  The conclusion of the theorem follows immediately.  

It remains to deal with the condition $(\star)$.  For each $x_0\in\CC^M$, observe that $P^{-1}(x_0)=:L_{x_0}$ is an affine plane of codimension $m$ while $E$ is an affine subvariety of  dimension at most $m-1$.  Hence we can find $x_0$ such that $L_{x_0}\cap E=\emptyset$. Then $$\bigcup_{\|x-x_0\|<\delta} L_x\cap E = \emptyset$$ for $\delta>0$ small enough; and in addition, 
$\left(\bigcup_{\|x-x_0\|<\delta} L_x \right)\cap V$ is a union of $d$ disjoint sets.  Now translate and rescale coordinates by the map $$(x,y)\longmapsto (\frac{1}{2\delta}(x-x_0),y)=(\tilde x,y).$$  
It is easy to see that $(\tilde x,y)$ is a Noether presentation for $V$ for which  $(\star)$ holds. We now consider the function $\log^+\|\tilde x\|$ and proceed as above.
\end{proof}

\section{Lelong degree} \label{sec:degree}
The complex Monge-Amp\`ere operator may be extended to certain unbounded functions.    
  If $u$ is a psh function on some domain $\Omega$, we define
$$
L(u):= \{z\in\Omega\colon u \hbox{ is \emph{not} locally bounded in } B(z,r)\cap\Omega,  \ \forall r>0\}.
$$
The following result (presented without proof) is a consequence of the general theory on complex manifolds developed in \cite{demailly:complex}, chapter III \S 4.  

\begin{proposition} \label{prop:41}
Let $T$ be a positive closed $(q,q)$-current in $\CC^N$ and let $u_1,\ldots,u_p$ be psh functions on $\CC^n$ with $p+q\leq n$.  Suppose that for each $j$, $L(u_j)$ is contained in an analytic set $A_j$.  
Given $m\in\{1,\ldots,p\}$, suppose that for each choice of $m$ functions $u_{j_1},\ldots,u_{j_m}$ the analytic set $A_{j_1}\cap\cdots\cap A_{j_m}$ has codimension at least $m$.  Then
\begin{enumerate}
\item The $(p+q,p+q)$-current $dd^cu_1\wedge\cdots\wedge dd^cu_p\wedge T$ is well defined on $\CC^n$.
\item If $\{u_j^{(k)} \}$ is a decreasing sequence of locally bounded psh functions on $\CC^n$, with $u_j^{(k)}\searrow u_j$ as $k\to\infty$, then
we have the weak-$*$ convergence of currents:
\begin{equation}\label{eqn:w*}
dd^cu_1^{(k)}\wedge\cdots\wedge dd^cu_p^{(k)}\wedge T \to dd^cu_1\wedge\cdots\wedge dd^cu_p\wedge T.
\end{equation} \qed
\end{enumerate} 
\end{proposition}

As an example of the second convergence, we have $$dd^c(\max\{\log|z_1|,-\log n\}) = [\{|z_1|=\tfrac{1}{n}\}]\wedge d\theta$$ by Proposition \ref{lem:43}, 
which converges to $[\{z_1=0\}]$. 

More generally, for an irreducible $p\in\CC[V]$ we have $\max\{\log|p|,-\log n\}\searrow \log|p|$ as $n\to\infty$, and we have the convergence
$$
dd^c\max\{\log|p(z)|,-\log n\}\wedge T = [\{ |p(z)|=\tfrac{1}{n} \}]\wedge d^c\log|p|\wedge T \to 2\pi[\{p(z)=0\}]\wedge T
$$
for any closed positive current $T$.  We can see this convergence near a regular point of $\{p=0\}$ by making a local holomorphic linearization of coordinates (which we denote by $H$) that transforms $\log|p|$ into $\log|z_1|$, and calculating as above.   
We recover the \emph{Lelong-Poincar\'e formula}: $$dd^c\log|p|\wedge T = 2\pi[\{p=0\}]\wedge T.$$ 


\bigskip

Let now $v\in\calL^+(V)$.  Suppose $u\in c\calL(V)$ for some $c>0$, and satisfies:
\begin{enumerate}
\item $L(u)\cap V$ is contained in an analytic set of dimension at most $m-1$; and
\item $\frac{1}{c}u$ is a decreasing limit $u_j\searrow\frac{1}{c}u$ of functions $u_j\in\calL^+(V)$.
\end{enumerate}
\begin{lemma}
Suppose $\{u_j\}_{j=1}^{\infty}$ and $u$ satisfy the two conditions above.  Suppose there is $v\in\calL^+(V)$ and a compact $K\subset V$ such that the support of $dd^c u_j\wedge (dd^cv)^{m-1}$ is contained in $K$ for each $j$.  Then
\begin{equation} \label{eqn:cL}
\int_V dd^c u\wedge (dd^cv)^{m-1}  =  \lim_{j\to\infty} c \int_V dd^cu_j\wedge(dd^cv)^{m-1} = c\deg(V)(2\pi)^m.
\end{equation} \end{lemma}

\begin{proof}
Let $\varphi$ be a smooth, compactly supported function on $V$ such that $\varphi|_{K}\equiv 1$.  Then the first equality is equivalent to
$$
\int_V \varphi dd^c u\wedge (dd^cv)^{m-1}  =  \lim_{j\to\infty} c \int_V \varphi dd^cu_j\wedge(dd^cv)^{m-1},
$$
which is true by (\ref{eqn:w*}). 

The second equality is a consequence of Corollary \ref{cor:L} and Theorem \ref{thm:L}.   
\end{proof}

We will apply the above lemma to $u=\log|p|$ where $p$ is a polynomial.  
To construct the $u_j$s, we will need another lemma.  Before stating it, we recall a standard fact: if $v\in\calL(\CC^m)$ is bounded from below and the Robin function
$$
\rho_v(z):=\limsup_{|\lambda|\to\infty} v(\lambda z)-\log|\lambda|
$$
is finite at every point of $\CC^m\setminus\{0\}$, then $u\in\calL^+(\CC^m)$.  

\begin{lemma}\label{lem:44}
Let $p\in\CC[x]$ be a polynomial and $\hat p$ its leading homogeneous part.  Suppose the set
\begin{equation}\label{eqn:44}
\{x\in\CC^m: \hat p(x)=x_2=\cdots=x_m=0\}
\end{equation}
consists only of the origin.  Let $j\in\NN$ and define
\begin{equation}\label{eqn:45}
u_j(x):=\max\{\tfrac{1}{\deg p}\log|p(z)|,\log\tfrac{|x_2|}{j},\ldots,\log\tfrac{|x_m|}{j},-\log j \}.
\end{equation}
Then $u_j\in\calL^+(\CC^m)$.
\end{lemma}
\begin{proof}
We have for each $x\in\CC^m$ that $\log\tfrac{|\lambda x_k|}{j} -\log|\lambda| = \log \tfrac{|x_k|}{j}$ and 
\begin{eqnarray*}
\lim_{|\lambda|\to\infty} \tfrac{1}{\deg p}\log|p(\lambda x)| -\log|\lambda|&=& \lim_{|\lambda|\to\infty} \tfrac{1}{\deg p}\log\left|\hat p(x) +O(|\lambda|^{-1})\right|  \\
&=& \tfrac{1}{\deg p}\log|\hat p(x)|, \end{eqnarray*} 
so $\rho_{u_j}(x) = \max\left\{\frac{1}{\deg p}\log|\hat p(x)|, \log\tfrac{|x_2|}{j},\ldots,\log\tfrac{|x_m|}{j}\right\}$.  By (\ref{eqn:44}), $\rho_{u_j}(x)$ is $-\infty$ only at the origin, and finite for all other $x\in\CC^m$. Hence $u_j\in\calL^+(\CC^m)$.  
\end{proof}

\begin{proposition} \label{prop:45}
Let $V\subset\CC^N$, with $\dim(V)=m$, and let $(x,y)$ be a Noether presentation of $V$ for which (\ref{eqn:44}) holds. Suppose $p\in\CC[x]\subseteq\CC[V]$.  Then for each $v\in\calL^+(V)$, 
\begin{eqnarray} \label{eqn:polyL}
\deg(p)\deg(V)(2\pi)^m &=& \int_V dd^c\log^+|p|\wedge(dd^cv)^{m-1}  \\
&=&\int_V dd^c\log|p|\wedge (dd^cv)^{m-1}.  \label{eqn:polyL2}
\end{eqnarray}
\end{proposition}

\begin{proof}
Let the functions $u_j$, $j\in\NN$, be as in the previous lemma.  We compute by Lemma \ref{lem:43},
\begin{eqnarray} 
(dd^cu_j)^m &=& [S_j]\wedge d^c\left(\tfrac{1}{\deg p}\log|p(x)|\right)\wedge d^c\log\tfrac{|x_2|}{j}\wedge\cdots\wedge d^c\log\tfrac{|x_m|}{j} \nonumber \\
&=& [S_j]\wedge  d^c\left(\tfrac{1}{\deg p}\log|p(x)|\right)\wedge d^c\log|x_2|\wedge\cdots\wedge d^c\log|x_m|. \label{eqn:46}
\end{eqnarray}
Let $T_j:=\{x\in\CC^m: |p(x)|\leq j^{-\deg p} \}$.  Then $S_j=T_j\cap S$, where $$S=\{|x_2|=\cdots=|x_m|=1\}.$$  
Put $v:=\max\{\log|x_2|,\ldots,\log|x_n|,0\}$; then we also have
\begin{equation*}
(dd^cv)^{m-1} = [S]\wedge d^c\log|x_2|\wedge\cdots\wedge d^c\log|x_m|;
\end{equation*}
in particular, $(dd^cv)^{m-1}$ is supported on $S$.  Together with (\ref{eqn:46}),  Lemma \ref{lem:44}, and Theorem \ref{thm:L}, we have
\begin{eqnarray*}
\int_{T_j} \varphi d^c\left(\tfrac{1}{\deg p}\log|p(x)|\right) \wedge (dd^cv)^{m-1} &=& \int_{S_j}  d^c\left(\tfrac{1}{\deg p}\log|p(x)|\right) \wedge (dd^cv)^{m-1} \\
 &=&  \int (dd^cu_j)^m = (2\pi)^m,
\end{eqnarray*}
where $\varphi$ on the left is a smooth compactly supported function with $\varphi|_{K}\equiv 1$; here we take 
$K:=\{|p(x)|\leq 1, |x_2|\leq 1, \ldots, |x_m|\leq 1\}$ which is a fixed compact set containing $S_j$ for all $j$.

For $j\geq 1$ we have $|p(z)|\leq 1$ so that $\log|p(z)|\leq 0\leq v(z)$ on $T_j$.  Hence $v|_{T_j} \equiv w|_{T_j}$, where $w=\max\{v,\tfrac{1}{\deg p}\log|p(z)|\}\in\calL^+(\CC^m)$, and the left-hand side is equal to
$$
\int_{T_j} \varphi  d^c\left(\tfrac{1}{\deg p}\log|p(x)|\right) \wedge (dd^cw)^{m-1}.
$$
Let $\{w_{\epsilon}\}_{\epsilon>0}$ be smooth functions in $\calL^+(\CC^m)$ with $w_{\epsilon}\searrow w$ as $\epsilon\to 0$.  Then the above integral is equal to 
\begin{equation}\label{eqn:epsilon}
\lim_{\epsilon\to 0} \int_{T_j}\varphi d^c\left(\tfrac{1}{\deg p}\log|p(x)|\right) \wedge (dd^cw_{\epsilon})^{m-1}.
\end{equation}
For fixed $\epsilon$, the integral inside the limit may be rewritten as the pairing of a current with smooth compactly supported form:
$$\left\langle [T_j]\wedge d^c\bigl(\tfrac{1}{\deg p}\log|p|\bigr), \varphi dd^cw_{\epsilon}\right\rangle
= \left \langle dd^c\max\bigl\{\tfrac{1}{\deg p}\log|p|,-\log j\bigr\},\varphi dd^cw_{\epsilon}   \right\rangle
$$
where we apply Lemma \ref{lem:43} with $k=1$; hence  (\ref{eqn:epsilon}) becomes 
\begin{eqnarray*}
&& \lim_{\epsilon\to 0} \int \varphi dd^c\max\{\tfrac{1}{\deg p}\log|p|,-\log j\}\wedge (dd^cw_{\epsilon})^{m-1} \\
&&\qquad\qquad =\ \int\varphi  dd^c\max\{\tfrac{1}{\deg p}\log|p|,-\log j\}\wedge (dd^cw)^{m-1}, 
\end{eqnarray*}
using convergence of currents (or Proposition \ref{prop:41}).  

Altogether, we have 
$$
(2\pi)^m \ = \ \int \varphi dd^c\max\{\tfrac{1}{\deg p}\log|p|,-\log j\}\wedge (dd^cw)^{m-1}.
$$
We may drop $\varphi$ from the above integral, as the support of $dd^c\max\{\log|p|,-\log j\} \wedge (dd^cw)^{m-1}$ is the set $K\cap\{|p|=\tfrac{1}{j}\}\subset K$.  
Finally, by Theorem \ref{thm:33}, we may replace $w$ with any other $v\in\calL^+(\CC^m)$ since the support of $dd^c\max\{\log|p|,-\log j\}$ is the (unbounded) set $\{|p|=\tfrac{1}{j}\}$.
 When $j=1$, this proves (\ref{eqn:polyL}) when $V=\CC^m$.   
Proposition \ref{prop:41} may be applied to take the limit as $j\to\infty$ on the right-hand side to yield 
$$
(2\pi)^m  \ = \ \frac{1}{\deg p}\int \varphi dd^c\log|p| \wedge (dd^cw)^{m-1},
$$
which  proves (\ref{eqn:polyL2}) when $V=\CC^m$.

For a general variety $V$, the same argument as in the proof of Theorem \ref{thm:L} gives 
\begin{eqnarray*}
\int_V  dd^c\log|p| \wedge (dd^cv)^{m-1} &=& \sum_{k=1}^{\deg V} \int_{\CC^m} dd^c\log|p| \wedge (dd^cv)^{m-1} 
\end{eqnarray*}
if $v=\log\|x\|$ and the support of the current $[V]\wedge dd^c\log|p| \wedge (dd^cv)^{m-1}$ is away from the branch locus of the projection.  Otherwise, as in that proof, we modify $v$ by an affine map so that the support of $(dd^cv)^{N-1}$ is away from the branch locus, and the above equation holds. Applying the $\CC^m$ case to each integral in the right-hand sum yields formula (\ref{eqn:polyL2}).  Similarly, we also get (\ref{eqn:polyL}).   Finally, by Theorem \ref{thm:33}, these formulas hold for all functions in $\calL^+(V)$.
\end{proof}

Proposition \ref{prop:45}  motivates the following definition.
\begin{definition}\rm
Fix $v\in\calL^+(V)$. Given $p\in\CC[V]$, the \emph{Lelong degree of $p$ on $V$} is defined by
$$
\deg_{\calL(V)}(p) \  := \  \frac{1}{\deg(V)(2\pi)^m} \int_V dd^c\log|p|\wedge (dd^cv)^{m-1}.
$$
\end{definition}

\begin{remark}\rm 
\begin{enumerate}
\item The Lelong degree is independent of $v$ by Corollary \ref{cor:L}, and  coincides with the usual degree when $p\in\CC[x]\subseteq\CC[V]$ by Proposition \ref{prop:45}. 
\item The definition in terms of a Monge-Amp\`ere formula is similar to that of a Lelong number:  $\deg_{\calL(V)}(p)$ may be interpreted as a Lelong number for the current $dd^c\log|p|$.  
\item The Lelong degree is independent of coordinates in $\CC^m$ since it is defined in terms of the $dd^c$-operator which is invariant under biholomorphic maps.  In particular, the definition makes sense without reference to a Noether presentation.  A Noether presentation is convenient for computation.
\end{enumerate}
\end{remark}

\begin{example}\rm
 Let $V\subset\CC^2$ be the quadratic curve with equation $x=y^2$.  Then $(x,y)$ is a Noether presentation. 
Since $\log|y|=\frac{1}{2}\log|x|$, we have 
$$
\deg_{\calL(V)}(y) = \tfrac{1}{2}\deg(x)=\tfrac{1}{2}
$$
by Proposition \ref{prop:45}.

For a general polynomial $p\in\CC[V]$, its normal form is $p_1(x)+yp_2(x)$.  Using the estimates
$$
\max\{|p_1|,|yp_2|\}\leq |p_1+yp_2| \leq 2\max\{|p_1|,|yp_2|\}
$$
we see that $\log|p|- \max\{\log|p_1|, \log|yp_2|\} = O(1)$.  Thus using Theorem \ref{thm:33}, we may replace $\log|p|$ by $u=\max\{\log|p_1|, \log|yp_2|\}$ in the computation of Lelong degree.  If $\deg(p_1)>\deg(p_2)$ then for sufficiently large values of $x$, we have 
$$\log|p_1(x)|>\log|p_2(x)| + \tfrac{1}{2}\log|x| = \log|yp_2(x)|,$$
so that $u=\log|p_1(x)|$.  On the other hand, if $\deg(p_1)\leq \deg(p_2)$ then $u=\log|yp_2(x)|=\tfrac{1}{2}\log|x|+\log|p_2(x)|$ for sufficiently large values of $x$.  It follows   that $$\deg_{\calL(V)}(p) = \max\{\deg(p_1),\deg(p_2)+\tfrac{1}{2}\}.$$
\end{example}

It may or may not be possible to replace $\log|p|$ by $\log^+|p|$ in the computation of Lelong degree.  We need an additional condition.

\begin{proposition}\label{prop:balayage}
Suppose $z=(x,y)$ is a Noether presentation of $V$ where $x=(x_1,\ldots,x_m)$.  Suppose $p\in\CC[V]$ satisfies one of the following conditions: for each $j=1,\ldots,m$, either the set 
$$K_j:=\{z\in V\colon |p(x,y)|\leq 1, \ |x_k|\leq 1  \hbox{ for all }k\neq j \}$$
is compact, or if $K_j$ is unbounded, then $\log|p|$ is $O(1)$ on $K_j$ as $|x|\to\infty$.  

  Then
$$
\deg_{\calL(V)}(p) = \frac{1}{\deg(V)(2\pi)^m}\int_V dd^c\log^+|p|\wedge(dd^cv)^{m-1}.
$$
More generally, we may replace $\log^+|p|$ in the integral by $\max\{\log|p|,c\}$ for any $c\in\RR$.  
\end{proposition}

\begin{proof}
Let us take the function $v:=\max\{\log|x_1|,\ldots,\log|x_m|,0\}\in\calL^+(V)$.    
Suppose the first condition holds.  Let $K:=\bigcup_{j=1}^m K_j$.  By Proposition \ref{lem:43}, the support of the current $dd^c\log^+|p|\wedge(dd^cv)^{m-1}$ is contained in 
	$K$. 
	
	Let $\varphi$ be a smooth, compactly supported function on $V$ such that $\varphi|_K\equiv 1$.  Then $dd^c\varphi|_K\equiv 0$, so that  
$$ 
 \int_V \log^+|p|dd^c\varphi\wedge(dd^cv)^{m-1} = \int_V \log|p|dd^c\varphi\wedge(dd^cv)^{m-1}.
$$
This is the same as $\int_V \varphi dd^c\log^+|p|\wedge(dd^cv)^{m-1} 
 =  \int_V \varphi dd^c\log|p|\wedge(dd^cv)^{m-1}$, and the result follows since we may replace $\varphi$ by $1$.

If $\log|p(z)|$ is $O(1)$ on $K_j$, then 
$$ \liminf_{\substack{|z|\to\infty\\ z\in K_j}} \log|p|\geq M+1$$ for some $M\in(-\infty,-1]$, so the set 
$$K_{j,M}:=\{z\in V: |p(x,y)|\leq e^{M}, |x_k|\leq 1 \hbox{ for all } k\neq j\}$$
is compact.  Hence $e^{-M}p$ satisfies the previous condition, so 
$$
\int_V dd^c\log|e^{-M}p|\wedge(dd^cv)^{m-1} = \int_V dd^c\log^+|e^{-M}p|\wedge(dd^cv)^{m-1}.
$$
We want to replace $e^{-M}p$ with $p$ on both sides.  We may do this on the left-hand side because
$$dd^c\log|e^{-M}p| = dd^c(\log|p| - M) = dd^c\log|p|.$$
On the right-hand side, $\log^+|e^{-M}p| - \log^+|p|=O(1)$, so we may apply Corollary \ref{cor:L2}.
\end{proof}	

\begin{example}\rm
Let $p(x,y) = x-y$ on the quadratic curve in $\CC^2$ given by $x^2-y^2=1$.  Then  $p(x,y)\to 0$ as $(x,y)$ tends to infinity along the curve in the direction of the asymptote $y=x$.  The set $K$ in the above proof is $\{|x-y|\leq 1\}$, which is unbounded, and $\log|p|\to -\infty$ on $K$ as $|x|\to\infty$.  So neither condition holds for $p$.  

The conclusion of the proposition also fails, by a calculation.  Write $w=x-y$, $v=x+y$; then $V$ is given by the equation $wv=1$.  We have
$$
\int_V dd^c\log^+|p| = \int_V dd^c\log^+|w| = \int_{\CC\setminus\{0\}} dd^c\log^+|w| = 2\pi.
$$
On the other hand, using the fact that $dd^c\log|w|=0$ on a neighborhood of $V\cap\{|w|\in(\tfrac{1}{2},2)\}$, we have
\begin{eqnarray*}
\int_V dd^c\log|p| &=& \int_{V\cap\{|w|<2\}} dd^c\log|w| \ + \  \int_{V\cap\{|w|>\tfrac{1}{2}\}} dd^c\log|w| \\
&=& \int_{V\cap\{|w|<2\}} dd^c\log|w| \  - \   \int_{V\cap\{|v|<2\}} dd^c\log|v|  \ = \  0.
\end{eqnarray*}
\end{example}

We close this section with the following result, used in the last section.

\begin{proposition}\label{prop:moving}
Let $V=\bV(p_1,p_2,\ldots,p_k)$ be an irreducible affine variety in $\CC^N$ of dimension $m=N-k$. Define $W\subset\CC^{N+1}$ by 
$$W:=\{(t,z)\in\CC^{N+1}\colon p_1(z)-t=p_2(z)=\cdots=p_k(z)=0\}=   \bV(p_1-t,p_2,\ldots,p_k),$$  
and define $V_t:=\bV(p_1-t,p_2,\ldots,p_k)$ when $t$ is fixed.  If $t$ is sufficiently small then $V_t$ is irreducible.  Moreover,  
$\deg(W)=\deg(V_t)=\deg(V),$ and  $$\deg_{\calL(W)}(p)=\deg_{\calL(V)}(p) = \deg_{\calL(V_t)}(p)$$ 
for any polynomial $p$ that does not depend on $t$.
\end{proposition}

\begin{proof}
 Restricting to a hyperplane (corresponding to fixing $t$ in the original coordinates), we obtain a Noether presentation for $V_t$ (or $V$ if $t=0$), and clearly $\deg(V_t)=\deg(W)$.  If $V$ is irreducible then so is $V_t$ for $t$ sufficiently small.
  
  Consider now, for $\epsilon>0$, the functions in $\calL^+(W)$ given by $$w_{\epsilon}(t,z):=\max\{\log|t|,\log|x_1|,\ldots,\log|x_m|,\log \epsilon\},$$ where $x$ are the independent variables in a Noether presentation of $V$, and let $v_{\epsilon}(z):=\max\{\log|x_1|,\ldots,\log|x_m|,\log \epsilon\}$. Then a calculation as in  
(\ref{eqn:37}) yields
\begin{eqnarray*}
(dd^cw_{\epsilon})^m &=& [\{|t|=\epsilon\}]\wedge d\theta(t)\wedge(dd^cv_{\epsilon})^{m-1}  \\
&=& dd^c(\max\{\log|t|,\log \epsilon\})\wedge (dd^cv_{\epsilon})^{m-1}. 
\end{eqnarray*}
 Considering $p\in\CC[V]$ as a polynomial in $\CC[W]$ which is independent of $t$, and writing $\tau_{\epsilon}=\max\{\log|t|,\log\epsilon\}$, we have 
\begin{eqnarray}
(2\pi)^{m+1}{\deg(W)}\deg_{\calL(W)}(p) &=& \int_W dd^c\log|p|\wedge (dd^cw_{\epsilon})^{m} \nonumber \\
&=& \int_{W} dd^c\log|p|
 \wedge (dd^cv_{\epsilon})^{m-1} \wedge dd^c\tau_{\epsilon}.  \label{eqn:47a}  
\end{eqnarray}
By Corollary \ref{cor:L2}, we can replace $\tau_{\epsilon}(z)$ on the right-hand side of (\ref{eqn:47a}) by $\tau_{\delta}(z)$ for any $\delta>0$, since $\tau_{\epsilon}-\tau_{\delta}=O(1)$.  Then
\begin{eqnarray*}
(2\pi)^{m+1}{\deg(W)}\deg_{\calL(W)}(p) &=& \int_{W} dd^c\log|p|
 \wedge (dd^cv_{\epsilon})^{m-1} \wedge dd^c\tau_{\delta} \\
&=& 2\pi\int_{V} dd^c\log|p|\wedge (dd^cv_{\epsilon} )^{m-1} \\
&=& (2\pi)^{m+1}\deg(V)\deg_{\calL(V)}(p)
\end{eqnarray*}
upon letting $\delta\to 0$ and using Monge-Amp\`ere convergence.  We showed above that $\deg(W)=\deg(V)$, hence $\deg_{\calL(W)}(p)=\deg_{\calL(V)}(p)$.
By the same type of argument, $\deg_{\calL(W)}(p)=\deg_{\calL(V_t)}(p)$ for any fixed $t\in\CC$.
\end{proof}

\section{Curves in $\CC^2$} \label{sec:curves}
The Lelong degree on an algebraic curve in $\CC^2$ may be computed from Puiseux series.  Let $V=\{(x,y)\in\CC^2\colon P(x,y)=0\}$, and suppose $(x_0,y_0)\in V$. Then by the theory of Puiseux series, there is a neighborhood $\Omega$ of $(x_0,y_0)$ in $\CC^2$ such that any point $(x,y)\in\Omega\cap V$  is given by a Puiseux series about $x_0$:
$$
y = \sum_{j=j_0}^{\infty} a_j (x-x_0)^{j/d}, \  \hbox{ where $j_0\in\ZZ$ and $d\in\NN$}.
$$
We apply Puiseux series (more precisely, the lowest term)  as follows.
\begin{proposition}
Suppose $V\subset\CC^2$ is an algebraic curve, with Noether presentation $(x,y)$.  There exists a $\gamma\in(-\infty,1]\cap\QQ$ such that if $R>0$ is sufficiently large, then
\begin{equation}\label{eqn:p1}
y=x^{\gamma}(c + o(|x|^{\gamma}))   \quad \hbox{as } |x|\to\infty, 
\end{equation}
and we take the limit along a continuous path in which $(x,y)\in V$ and $|x|>R$.  As a consequence,
$
\log|y| = \gamma\log|x| + O(1)$ 
along such a path.
\end{proposition}

\begin{proof}
Let $V=\{P(x,y)=0\}$.  Under the change of coordinates $$(x,y)\mapsto (1/x,y/x)=:(s,t),$$  for each $(x,y)$ with $x\neq 0$ we have $P(x,y)=0$ if and only if $\tilde P(s,t)=0$, where $\tilde P$ is the polynomial obtained by replacing each term $a_{jk}x^{j}y^k$ with $a_{jk}s^{j'}t^k$, where $j'=\deg(P)-j-k$.   

Consider a Puiseux series at the origin that gives $t$ in terms of $s$ when $\tilde P(s,t)=0$:
$$
t  \ = \  \sum_{j=j_0}^{\infty} a_js^{j/d} \  = \   a_{j_0}s^{j_0/d} + o(|s|^{j_0/d}) \  \hbox{ as } |s|\to 0;
$$
changing back to affine coordinates and multiplying through by $x$ yields 
$$
y =  a_{j_0}x^{1-j_0/d} + o(|x|^{1-j_0/d}) \ \hbox{ as } |x|\to\infty.
$$
For sufficiently large $|x|$, this says that $|y|>\tfrac{1}{2}|a_{j_0}||x|^{1-j_0/d}$.  Since $(x,y)$ is a Noether presentation, the estimate in (\ref{eqn:yx}) must hold, so $1-j_0/d\leq 1$.  
We set $\gamma:=1-j_0/d$.  The Puiseux series in $s$ is valid on an open set for which the origin is a limit point. A continuous path in this open set with $s\to 0$ corresponds to a continuous path in $x$ with $|x|\to\infty$.   
\end{proof}

The value of $\gamma$ in the above proposition is easy to read off from the Newton polygon associated to $\tilde P(s,t)$.\footnote{It is the first step in an iterative algorithm for computing the terms of a Puiseux series.  See e.g. \cite{willisdidiersonnaburg:how} for a description.}
We illustrate with the following example.

\begin{example}\label{ex:52} \rm
Let $V\subset\CC^2$ have defining polynomial $$P(x,y)=y^4 + xy^3+xy^2+x^2y+x^2+1.$$
The associated polynomial is $\tilde P(s,t)=t^4+t^3+st^2+st+s^2+s^4$.  

\bigskip

\begin{figure}

\begin{multicols}{2}

\includegraphics[scale=0.8]{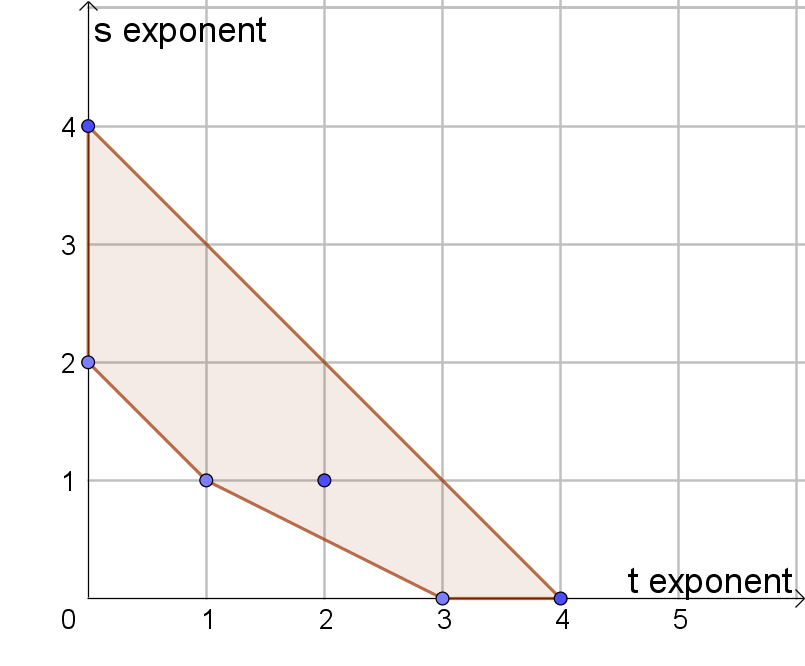}

\includegraphics[scale=0.8]{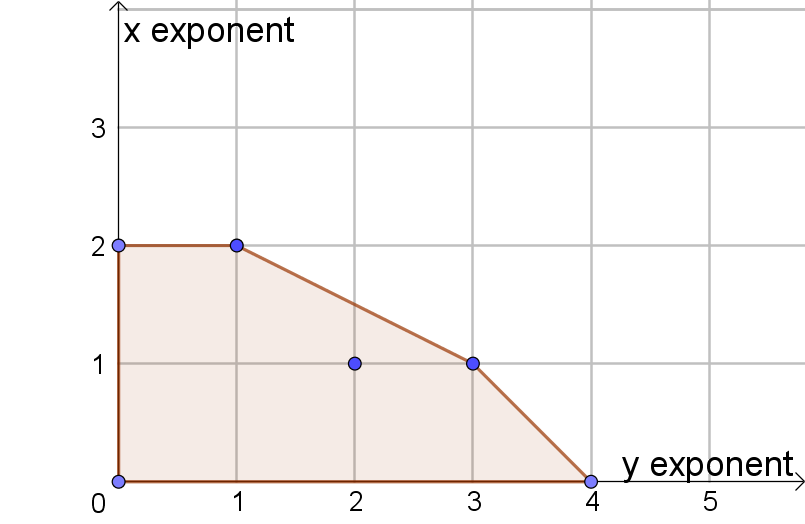}

\caption{Newton polygons of $P$ (above) and $\tilde P$ (left) in Example \ref{ex:52}.}

\end{multicols}
\end{figure}

Write a Puiseux series for $t$ in terms of $s$ as $t = \sum_{k=1}^{\infty} c_ks^{\gamma_k}$.  Then
the lowest power $\gamma_1$ is the negative of a slope of a so-called \emph{lower segment} of the boundary of the Newton polygon (i.e. any downward translate of a lower segment gives a line segment that does not intersect the polygon).

\bigskip

In our case, the polygon $N$ has three lower segments  (see Figure 1, left). 

\begin{enumerate}
\item Line segment joining $(3,0)$ and $(4,0)$, with slope $0$.  Here $\gamma_1=0$, and we write the Puiseux series as 
$t = c_1 + t_1(s) = c_1 + c_2s^{\gamma_2}+\cdots$.  Put this into the equation for the curve (write $c:=c_1$ for convenience):
\begin{eqnarray*} 
0=P(s,c+t_1(s)) &=& (c+t_1(s))^4+(c+t_1(s))^3+s(c+t_1(s))^2+s^2+s^4 \\
&=& (c^4 + c^3) + \hbox{ (higher terms in $s$)}.
\end{eqnarray*}
Equating coefficients in the constant term, we obtain $0=c^4+c^3$ which is a nonzero coefficient when $c=-1$.  Hence $t=-1+o(1)$ as $s\to 0$, which becomes 
$$
y=-x+o(|x|) \quad \hbox{as } |x|\to\infty.
$$

\item Line segment joining $(3,0)$ and $(1,1)$.  Here $\gamma_1=\tfrac{1}{2}$, so that $t=cs^{1/2}+ t_1(s)$.  Equating the lowest nonzero coefficient in the Puiseux series equation for $P(s,t)=0$, we obtain $0=t^{3/2}(c^3+c)$ so that $c=\pm i$.  Hence we obtain two more series: $t=is^{1/2}+ o(|s|^{1/2})$, and $t=-is^{1/2} +o(|s|^{1/2})$, giving
$$
y=ix^{1/2}+o(|x|^{1/2}) \hbox{ or } y=-ix^{1/2}+ o(|x|^{1/2}) \quad \hbox{as }|x|\to\infty.
$$

\item Line segment joining $(1,1)$ and $(0,2)$.  Here $\gamma_1=1$, so that $t=cs + t_1(s)$.  Solving for the lowest nonzero coefficient in $P(s,t)=0$, we obtain $t^2(c+1)=0$, so that $c=-1$. At the end we obtain
$$
y=-1+o(1) \quad\hbox{as }|x|\to\infty.
$$
\end{enumerate}
\end{example}

For this example, $(x,y)$ is a Noether presentation because $P(x,y)$ is of the form (\ref{eqn:gj}), and $\deg(V)=4$ since generically there are $4$ solutions in $y$ for fixed $x$.  The four branches of $V$ over $\CC$ correspond to the four Puiseux series. Write this as 
$$V = V_1 \cup V_{\frac{1}{2},1} \cup V_{\frac{1}{2},2} \cup V_0,$$
where the (first) subscript corresponds to the exponent $\gamma$ in (\ref{eqn:p1}).  

 For points $x$ where $|x|>R$ and $R$ is sufficiently large, the branches associated to Puiseux series with different values of $\gamma$ are disjoint, because the corresponding $y$ values are of a different order of magnitude. In particular, $V_1$ and $V_0$ are isolated branches for large $|x|$.  
The two branches $V_{\frac{1}{2},1}, V_{\frac{1}{2},2}$ have $y$ values of the same order of magnitude for large $|x|$.  Write $V_{\frac{1}{2}}:= V_{\frac{1}{2},1}\cup V_{\frac{1}{2},2}$; then for large $|x|$, $V_{\frac{1}{2}}$ is approximated up to $o(|x|^{1/2})$ by the curve 
$y^2+x=0$, obtained by discarding the lower order terms in both series.  We will denote this curve by $W_{\frac{1}{2}}$.  See Figure 2.

\begin{figure}
\begin{multicols}{2}
\includegraphics[height=6cm]{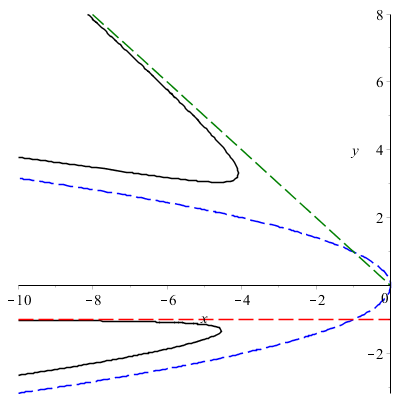}
\caption{  The real points of $V$ are given by the two solid curves.  The three dashed curves approximate $V_0,V_{\tfrac{1}{2}},V_{1}$.   The next terms in the series for $\gamma=0$ and $\gamma=1$ are $o(1)$, so the dashed lines approximating $V_0,V_1$   actually converge to these branches.  For $\gamma=\frac{1}{2}$, the next term in the Puiseux series for each branch of $V_{\frac{1}{2}}$ is $O(|x|^{1/4})$.  In this case, the branches of $W_{\frac{1}{2}}$  follow $V_{\frac{1}{2}}$ but separate from it very slowly.    Plotted with Maple 18.}
\end{multicols}
\end{figure}

Let $V(R):=V\cap\{|x|>R\}$, and similarly for other sets.  We evaluate $\log|y|$ on these three pieces of $V(R)$:
$$
\log|y| = \left\{\begin{array}{rlcrl}
\log |-x+o(x)| &\hbox{on } V_1(R)  & = &  \log|x|+O(1) &\hbox{on } \CC(R) \\
\log|-1 + o(1)| &\hbox{on } V_0(R) & = & O(1) &\hbox{on } \CC(R)  \\
\log|y| &\hbox{on } V_{\frac{1}{2}}(R) & = &  \log|\tilde y| + O(1) &\hbox{on } W_{\frac{1}{2}}(R)
\end{array} \right.  .
$$
Here, $\tilde y$ is just given by the closest value on the approximating curve: if $(x,y)\in V_{\frac{1}{2}}$ then 
$(x,\tilde y)\in W_{\frac{1}{2}}$.

We now compute the Lelong degree of $y$. 
Clearly $V\setminus V(R)$ is compact, by (\ref{eqn:yx}).  Set 
$C>\sup\{|y|:  (x,y)\in V\setminus V(R)\}$, so that $dd^c\max\{\log|y|, \log C\}$ is supported on the set $\{|y|= C\}\subset V(R)$.  Then 
\begin{eqnarray*}
\int_V dd^c\max\{\log|y|,C\} &=& \int_{V(R)} dd^c\max\{\log|y|,C\} \\
&=& \int_{V_1(R)} dd^c\max\{\log|y|,C\} + \int_{V_0(R)} dd^c\max\{\log|y|,C\} \\
&& \hskip3cm + \int_{V_{\frac{1}{2}}(R)} dd^c\max\{\log|y|,C\} \\
&=& \int_{\CC(R)} dd^c\max\{\log|x|+O(1),\log C\} + \int_{\CC(R)} dd^c \log C \\
&& \hskip2cm + \int_{W_{\frac{1}{2}}(R)} dd^c\max\{\log|y|+O(1),\log C\} \\
&=& 2\pi \ +   \   0  \ + 2\pi \ = \ 4\pi,
\end{eqnarray*}
where we use the fact that $\log|y|=\tfrac{1}{2}\log|x|$ on $W_{\frac{1}{2}}$.  
Hence $\displaystyle\deg_{\calL(V)}(y) = \frac{4\pi}{4(2\pi)} = \tfrac{1}{2}$.

\begin{remark} \label{rmk:lelong} \rm
\begin{enumerate}
\item The end result of the computation is to take the average growth of $\log|y|$ over all four branches:
\begin{equation}\label{eqn:avg}
\deg_{\calL(V)}(y) = \tfrac{1}{4}(1+\tfrac{1}{2}+\tfrac{1}{2}+ 0) = \tfrac{1}{2}.
\end{equation}
\item The additional condition of Proposition \ref{prop:balayage} was needed to simplify the computation of the integral, replacing $\log|y|$ by $\max\{\log|y|,C\}$ and projecting to $\CC^m$. The condition is easily seen to hold if all Puiseux series exponents are nonnegative. If one of the exponents is negative then the Monge-Amp\`ere computation fails; however, the averaging formula (\ref{eqn:avg}) should still hold with the appropriate sign.  For example, $y^3-x^2y+x+1=0$ is a degree 3 curve with the 3 series
$$
y=x+\cdots,\quad y= -x + \cdots,\quad y= x^{-1} + \cdots,
$$
and we should have $\deg_{\calL(V)}(y) = \tfrac{1}{3}(1+1-1) = \tfrac{1}{3}$.

\item Clearly, the minimum nonnegative constant $d$ permitting the inequality $$\tfrac{1}{d}\log|p(z)|\leq \log^+|z|+A$$ for some $A\in\RR$ and all $z\in V$ is given by the maximum value of $\gamma$. (For Example \ref{ex:52},   $d=1$.)

\end{enumerate}
\end{remark}

\subsection{Lelong degree formula for $y$}

Let $V=\bV(P)\subset\CC^2$, where $P$ is an irreducible polynomial of the form 
\begin{equation}\label{eqn:52}
P(x,y)=y^{\deg P}+ q(x,y).\end{equation}
This condition ensures that $(x,y)$ is a Noether presentation.  Also, since $p$ is irreducible, $q$ contains a nonzero term in $x^k$, for some $0\leq k\leq \deg(P)$.  Let $d$ be the maximum such value.  Then
\begin{equation}\label{eqn:lelongdegree}
\deg_{\calL(V)}(y) = \frac{d}{\deg(P)}.
\end{equation}
In Example \ref{ex:52}, the highest power of $P$ in $x$ alone is  $x^2$.  Hence by the above formula,  $\deg_{\calL(V)}(y)=\tfrac{2}{4}=\tfrac{1}{2}$.

Let us describe how formula (\ref{eqn:lelongdegree}) arises, using Example \ref{ex:52} as an illustration.  First, the Newton polygon $N$ of $\tilde P(s,t)$ has 3 lower segments, the negative of whose slopes are $\gamma_1=0$, $\gamma_2=\frac{1}{2}$, and $\gamma_3=1$.  The corresponding growths of $y$ in terms of $x$ as $|x|\to\infty$ are then calculated to be 
$$
c_1:=1-\gamma_1=1,\ c_2:=1-\gamma_2=\tfrac{1}{2},\ c_3:=1-\gamma_3=0.
$$
Let us relate this to the Newton polygon of $P(x,y)$, which we denote by $N_0$.  The lower segments of $N$ correspond to \emph{upper} segments of $N_0$ (i.e., that separate from $N_0$ when translated up).  The values of $c_j$ for each $j$ are precisely the negatives of the slopes of the segments in $N_0$.  (Compare the two polygons in Figure 1.)

In general, if $c_j=\frac{m_j}{n_j}$ (in simplest form) as $|x|\to\infty$, we get an approximation to $y^{n_j}=ax^{m_j}$ (for some $a$), a curve which has $n_j$ branches over $x$. By an argument using continuity, the curve ought to be approximating $n_j$ such branches of $V$ (in Example \ref{ex:52}, two branches for $c_2=\tfrac{1}{2}$).  As every branch for $c_j$ is associated to a collection of $n_j$ branches given by an approximating curve, the total number of branches must be a integer multiple, say $q_jn_j$.  Altogether, $\deg(V)=\sum q_jn_j$.

 In view of Remark \ref{rmk:lelong}(1) and the previous paragraph, we obtain
$$
\deg_{\calL(V)}(y) = \frac{1}{\deg(V)}\sum q_jn_j\tfrac{m_j}{n_j} = \frac{1}{\deg(V)}{\sum q_jm_j} = \frac{\sum q_jm_j}{\sum q_jn_j}.
$$
Now $\sum q_jn_j=\deg(V)=\deg(P)$ is the horizontal length of the Newton polygon, and $m_j/n_j$ are the negatives of the slopes of the upper segments (indexed by $j$), which we can consider as forming the graph of a piecewise linear function on the interval $[0,\deg(P)]$.  Integrating the slopes over this interval gives $\sum q_jm_j$, which gives the total decrease in height of the function. Since the graph starts at $(0,d)$ and ends at $(\deg(P),0)$, we get $d$.

\section{A B{\'e}zout theorem for affine varieties}

Let $V\subset\CC^N$ be an irreducible affine variety of dimension $m\geq 1$ given by $N-m$ polynomials $p_1,p_2,\ldots,p_{N-m}\in\CC[z]$.  Then 
$$
V=:V_m\subset V_{m+1}\subset\cdots\subset V_{N-1}\subset V_N:=\CC^N
$$
where $V_j$ is given by the polynomials $p_1,\ldots,p_{N-j}$ for $j=m,\ldots,N-1$.  

\begin{theorem}
Suppose $V$ is irreducible.  Then 
$$
\deg(V) = \prod_{j=m}^{N} \deg_{\calL(V_{N-j})}(p).
$$
\end{theorem}

\begin{proof}
Irreducibility of $V_{m+1},\ldots,V_N$ follows from irreducibility of $V$ by induction: if $V_{j+1}=V_{j+1,1}\cup V_{j+1,2}$, then $V_j=W_1\cup W_2$, where $W_k= \bV(p_{N-j})\cap V_{j+1,k}$, contradicting the inductive hypothesis that $V_j$ is irreducible.  We also have $\dim(V_{j+1})-\dim(V_{j})\in\{0,1\}$; since $\dim(V_m)=m$ and $\dim(V_N)=N$, the difference in dimension must be 1 throughout, showing that $\dim(V_j)=j$.

Fix $j$, and write $V_{j}=V_{j+1}\cap \bV(p_{N-j})$.  Let $v$ be a smooth function in $\calL^+(V_{j+1})$.  Then 
\begin{eqnarray*}
\deg_{\calL(V_{j+1})}(p_{N-j}) &=& \frac{1}{(2\pi)^{j+1}\deg(V_{j+1})}\int_{V_{j+1}} dd^c\log|p_{N-j}|\wedge (dd^c v)^{j} \\
&=& \frac{1}{(2\pi)^{j+1}\deg(V_{j+1})} \int_{V_{j}} 2\pi (dd^cv)^j \\
&=& \frac{\deg(V_j)}{\deg(V_{j+1})}.
\end{eqnarray*}
Hence
$$
\prod_{j=m}^{N} \deg_{\calL(V_{j+1})}(p_{N-j}) \ = \ \prod_{j=m}^N \frac{\deg(V_j)}{\deg(V_{j+1})} \ = \frac{\deg(V)}{\deg(\CC^N)} \ = \ \deg(V). 
$$
\end{proof}


\begin{corollary}[Affine B\'ezout theorem]
Let $p_1,\ldots,p_N\in\CC[z]$ be polynomials, each of degree $\geq 1$.  Suppose $V:=\bV(p_1,\ldots,p_N)$ is finite and $V_1:=\bV(p_2,\ldots,p_{N-j})$ is irreducible. 
  Then the number of points of $V$ is at most 
$$
\prod_{j=1}^N \deg_{\calL(V_{N-j})}(p_{j}).
$$  
\end{corollary}

\begin{proof}
Let $W:=\bV(p_1(z)-t,p_2(z),\ldots,p_N(z))\subset\CC^{N+1}$.  Then $W$ is an irreducible curve for generic values of $t$.   By the previous result, 
$$
\deg(W) = \prod_{j=1}^N \deg_{\calL(W_{j+1})}(p_{N+1-j}) = \deg(p_1)\prod_{j=1}^{N-1} \deg_{\calL(V_{j})}(p_{N-j}),
$$
where we use $\deg_{\calL(W_{N+1})}(p_1)=\deg(p_1)$ (since $W_{N+1}=\CC^{N+1}$ by our convention), and $\deg_{\calL(W_{j+1})}(p_{N-j}) =\deg_{\calL(V_j)}(p_{N-j})$ by  Proposition \ref{prop:moving}.

For $\epsilon>0$, let $L_{\epsilon}:\CC^{N+1}\to\CC^{N+1}$ be a linear map close to the identity ($\|L_{\epsilon}- I\|\to 0$ as $\epsilon\to 0$), such that $(t_{\epsilon},z_{\epsilon})=L_{\epsilon}(t,z)$  is a Noether presentation for $W$.  For any value of $t_{\epsilon}$ away from the branch locus of the projection $\pi$ to the first coordinate,  we have local inverses  $\zeta_1,\ldots,\zeta_D$ of $\pi$,  where $D=\deg(W)$.  Pick one of these (say $\zeta_1$) and define $\varphi_{\epsilon}:= \pi\circ L_{\epsilon}^{-1}\circ\zeta_1$.  Then $\varphi_{\epsilon}$ is holomorphic and locally invertible for $\epsilon$ sufficiently small (since $\varphi_{\epsilon}'\to  1$ locally uniformly as $\epsilon\to 0$).  

For each $j=1,\ldots,D$, the map $L_{\epsilon}^{-1}\circ\zeta_j\circ \varphi_{\epsilon}^{-1}$ gives a local inverse for $\pi$ in a neighborhood of $t$ in the original coordinates, $(t,x)\stackrel{\pi}{\mapsto} t$.  For fixed $x$ these local inverses give $D$ points $(t,x)$, for generic values of $t$, i.e., $\deg(W)$ points of $\bV(p_1-t,p_2,\ldots,p_N)$.  Letting $t\to 0$, a continuity argument gives $\deg(W)$ as an upper bound for the number of points in $\bV(p_1,\ldots,p_N)$.
\end{proof}

\begin{remark}\rm
\begin{enumerate}
\item To remove the condition that $V_1$ is irreducible, we can extend the definition of Lelong degree to unions of affine varieties of the same dimension.  We then treat each component in the above proof separately, and sum over all components at the end. 
\item We can introduce the notion of multiplicity of a point to get a formula with equality.  The point $b$ is of multiplicity $m$ if there is $\epsilon_0>0$ such that for every $\epsilon\in(0,\epsilon_0)$ there exists $\delta>0$ such that for a generic choice of $t=(t_1,\ldots,t_N)\in\CC^N$ with $\|t\|<\delta$, the set $$\bV(p_1-t_1,\ldots,p_N-t_N)\cap\{|z|<\epsilon\}$$ consists of exactly $m$ points.  In the above proof, $b$ is of multiplicity $m$ if there are $m$ distinct points $(t,b_{1,t}),\ldots,(t,b_{m,t})$ of $W$ that coalesce into $(0,b)$ as $t\to 0$ along a generic path.
\end{enumerate}
 \end{remark}

\begin{example}\rm
We illustrate with a simple example in $\CC^2$: compute the number of points of $V=\bV(p_1,p_2)$, where
\begin{equation}\label{eqn:ex55}
p_1(x,y):=y^4+xy^3+xy^2+x^2y+x^2+1,\quad p_2(x,y):=y-1.
\end{equation}

\begin{enumerate}
\item {\bf B\'ezout's theorem in $\PP^2$}.  Homogenizing in the variable $t$ gives
\begin{equation}\label{eqn:ex55a}
y^4+xy^3+xy^2t + x^2yt+ x^2t^2 + t^4=0, \quad  y-t=0.
\end{equation}
The intersection of a degree $4$ curve with a degree $1$ curve is $4\cdot 1=4$ points (counting multiplicity), by B\'ezout's theorem.  To get the affine points, we must discard  points at infinity: putting $t=0$ in (\ref{eqn:ex55a}) gives the equations $y^4+xy^3=0$ and $y=0$,  yielding the point $[t:x:y]=[0:1:0]$.   This is a point of multiplicity 2 as can be seen as follows.  First,   dehomogenize in the variable $x$ to get local coordinates at infinity: setting $x=1$ in (\ref{eqn:ex55a}) gives
$$y^4+y^3+y^2t+yt+t^2+t^4=0, \quad y=t.$$
Substituting the second equation into the first yields $$0=2t^2+2t^3+2t^2=2t^2(t^2+t+1),$$ so $t=0$ is a root of multiplicity $2$.  

Discarding the double point at infinity leaves 2 affine points of $V$.

\item {\bf Affine B\'ezout theorem with $V_2=\bV(p_2)$}.  We have $\deg(p_2)=1$.  Next, $(x,y)$ is a Noether presentation of $V_2$; plugging $y-1=0$ into $p_1$ gives $2x^2+2x+2$ as its normal form, so $\deg_{\calL(V_2)}(p_1)=2$ by Proposition \ref{prop:45}.  We obtain $1\cdot 2 = 2$ points.

\item {\bf Affine B\'ezout theorem with $V_2=\bV(p_1)$.}  We have $\deg(p_1)=4$.  Now $\deg_{\calL(V_2)}(p_2)=\deg_{\calL(V_2)}(y-1)$.  The change of coordinates $(x,y)\mapsto (x,y+1)$ takes $y-1$ to $y$ and takes $P$ to a polynomial whose highest exponent in $x$ alone is still $x^2$, so applying (\ref{eqn:lelongdegree}) in these new coordinates gives the same result. Thus $\deg_{\calL(V_2}(y-1) = \tfrac{1}{2}$, and we obtain $4\cdot\tfrac{1}{2} = 2$ points, as before.
\end{enumerate}

\end{example}

\bibliographystyle{plainnat}
\bibliography{myreferences}

\end{document}